\documentclass[review,onefignum,onetabnum]{siamart171218}
\usepackage[numbers,sort&compress]{natbib}

\usepackage{lipsum}
\usepackage{amsfonts}
\usepackage{graphicx}
\usepackage{epstopdf}
\usepackage{algorithmic}
\usepackage{booktabs}
\usepackage{amsmath}
\ifpdf
  \DeclareGraphicsExtensions{.eps,.pdf,.png,.jpg}
\else
  \DeclareGraphicsExtensions{.eps}
\fi

\newsiamremark{remark}{Remark}
\newsiamremark{example}{Example}
\newsiamremark{property}{PROPERTY}
\newsiamremark{hypothesis}{Hypothesis}
\crefname{hypothesis}{Hypothesis}{Hypotheses}
\newsiamthm{claim}{Claim}

\headers{M\"untz Sturm-Liouville Problems: Theory and Numerical Experiments}{Hassan Khosravian-Arab  and Mohammad Reza Eslahchi}

\title{M\"untz Sturm-Liouville Problems: Theory and Numerical Experiments}

\author{Hassan Khosravian-Arab\thanks{Department of Applied Mathematics, Faculty of Mathematical Sciences, Tarbiat Modares University, P.O. Box 14115-134, Tehran, Iran 
		(\email{h.khosravian@aut.ac.ir,\ h.khosravian@modares.ac.ir}, \email{eslahchi@modares.ac.ir}).}
	\and M. R. Eslahchi\footnotemark[1]}

\usepackage{amsopn}

\ifpdf
\hypersetup{
  pdftitle={M\"untz Sturm-Liouville Problems: Theory and Numerical Experiments},
  pdfauthor={Hassan Khosravian-Arab  and Mohammad Reza Eslahchi}
}
\fi

\nolinenumbers
\begin{document}

\maketitle

\begin{abstract}
  This paper presents two new classes of M\"untz functions which are called Jacobi-M\"untz functions of the first and second types. These newly generated functions satisfy in two self-adjoint fractional Sturm-Liouville problems and thus they have some spectral properties such as: orthogonality, completeness, three-term recurrence relations and so on. With respect to these functions two new orthogonal projections and their error bounds are derived. Also, two new M\"untz type quadrature rules are introduced. As two applications of these basis functions some fractional ordinary and partial differential equations are considered and numerical results are given.
\end{abstract}



\section{Introduction}
Roughly speaking, the usual {\it Spectral Methods} such as: Galerkin, Tau, Petrov-Galerkin, pseudo spectral and collocation methods, have attracted  attentions of many researchers in the field of numerical analysis. They have commonly used to solve various problems in engineering and sciences numerically. The main advantages of these methods are, in fact, simple implementations, rapid  and high accuracy approximations ({\it exponential accuracy or spectral accuracy}) for the smooth functions. The exponential accuracy of the usual spectral methods for the smooth solutions comes from the fact that in these methods the (smooth) solution is expanded in terms of the (orthogonal) polynomial basis functions (are so-called as the trial functions) which  are, in fact, sufficiently smooth functions. On the other hand, we have:
\begin{equation*}
\left. \begin{tabular}{l}
\text{Smooth solutions} \\
\ \  \ \ \ \ \ \ \ \ \ {\large\bf{+}}\\
\text{Smooth trial functions} 
\end{tabular}
\right \} \Longrightarrow  \ \text{Spectral methods with exponential accuracy}.
\end{equation*}
 Here, an important question remains to be answered is: how can we solve the problems with non-smooth solutions (especially, singularity at the end points of their domains) by the spectral methods with exponential accuracy? To answer the question, we need to develop the theory of the usual spectral methods for the non-smooth trial functions. To do so, we need to use the definition of the M\"untz functions. 

As we are aware, a M\"untz sequence is an increasing sequence of real numbers: ${\bf\Lambda}:=\{\lambda_k\}_{k=0}^\infty,\ \lambda_0<\lambda_1<\ldots\ $. With respect to the M\"untz sequence, a M\"untz system is used for a system of the form $(x^{\lambda_0},x^{\lambda_1},\ldots)$ and also the corresponding M\"untz space associated with ${\bf\Lambda}$ is also defined as:
\begin{equation*}
\mathcal{M}({\bf\Lambda}):= \bigcup\limits_{n=0}^\infty\mathcal{M}_n({\bf\Lambda})=\text{span}\{x^{\lambda_n},\ n=0,1,\ldots\},\ x\in[0,1],
\end{equation*}
where $\mathcal{M}_n({\bf\Lambda}):=\text{span}\{x^{\lambda_0},x^{\lambda_1},\ldots,x^{\lambda_n}\}$ for each $n=0,1,\ldots$.

It is easy to verify that when $0<\lambda_0<1$, then:
 \begin{equation}\label{Eq_3}
 u\simeq u_N=\sum_{k=0}^{N}c_kx^{\lambda_k},\ \lambda_0<\lambda_1<\ldots<\lambda_Nو
 \end{equation}
 is non-smooth at $x=0$. Therefore, when we have a problem with non-smooth solution at $x=0$, it is better to expand the solution in the form \eqref{Eq_3}. 
 
It should be noted that the above expansion is known as  M\"untz expansion  and the density of the M\"untz expansion on $C[0,1]$ for $\lambda_{0}=0$  is guarantied if  and only if  $\displaystyle \sum_{k=1}^{\infty}\frac{1}{\lambda_k}=\infty$ \cite{MR2350268}. 

Unfortunately,  in practice, the M\"untz expansion \eqref{Eq_3} have some difficulties from the numerical analysis point of view. One of the most important  difficulties is that the Fourier coefficients $c_k$ in \eqref{Eq_3}  quickly become unmanageably large \cite{MR3531734,MR2824680}. Perhaps, if it is possible, one of the most interesting ideas to eliminate these difficulties is to rewrite expansion \eqref{Eq_3} in terms of classical or modified orthogonal polynomials (or functions). So it is natural to figure out how to extract the classical or modified orthogonal M\"untz functions.
 
 It is known that some classes of orthogonal polynomials, which are so-called as the classical orthogonal polynomials, such as: Jacobi ( and its special cases: Chebyshev (first
and second kinds), Legendre, Gegenbauer), Laguerre and Hermite  can be obtained directly from the following second order ordinary differential equation which is called the Classical Sturm-Liouville Problems (CSLPs):
\begin{equation*}
\frac{d}{dx}\left(\rho(x)y'(x)\right)=\lambda_n \omega(x)y(x),
\end{equation*}
under some suitable boundary conditions. So, it is very natural that the main target of many researchers is to establish spectral methods with exponential accuracy by deriving some extended forms of the CSLPs to obtain other new classes of orthogonal polynomials (or functions) with outstanding features. 

 The review of the existing literatures clearly indicates that there are a number of different approaches to formulate various forms of the CSLPs. Due to the increasing growth of applications of fractional derivatives (and integrals) in many fields of sciences and engineering,   undoubtedly, one of the most interesting approaches to extend the CSLPs can be obtained when we replace the ordinary derivatives in CSLPs with fractional ones. These  formulations are so-called as fractional Sturm-Liouville problems (FSLPs). However, we can classify some of the most important forms of the  FSLPs  into the following six categories \cite{Khosravian-Arab2015526}:
 
\begin{itemize}
	\item The first formulation of  FSLPs is in the following form:
\begin{equation}\label{FSLF1}
\mathcal{D}^\alpha(p(x)y'(x))+q(x)y(x)=\lambda \omega(x)y(x),\ x\in[a,b],
\end{equation}
where $\alpha\in(0,1)$.
\item The second formulation of  FSLPs is as:
\begin{equation}\label{FSLF2}
\frac{d}{dx}\left(p(x)\mathcal{D}^\alpha y(x)\right)+q(x)y(x)=\lambda \omega(x)y(x),\ x\in[a,b],
\end{equation}
where $\alpha\in(0,1)$.
\item  The third formulation of  FSLPs is:
\begin{equation}\label{FSLF3}
\mathcal{D}^{\alpha}(p(x)y(x))+q(x)y(x)=\lambda \omega(x)y(x),\ x\in[a,b],
\end{equation}
where $\alpha\in(1,2)$.

We also point out that in the above three formulations $\mathcal{D}^{\alpha}$ is in the Riemann-Liouville or Caputo senses.
\item The fourth type of the FSL formulation is as follows \cite{Klimek2013795,Zayernouri2013495}:
\begin{equation}\label{FSLF5}
{}^{\pm}\mathcal{D}^{\alpha}(p(x){}^{\mp}\mathcal{D}^{\alpha}y(x))+q(x)y(x)=\lambda \omega(x)y(x),\ x\in[a,b],
\end{equation}
where $\alpha\in(0,1)$ or $\alpha\in(\frac12,1)$.
\item The fifth type of the FSL formulation is in the following form:
\begin{equation}\label{FSLF4}
\frac{d}{dx}\left(\rho(x)y'(x)\right)+\left({}^{\pm}\mathcal{D}^{\alpha}+{}^{\mp}\mathcal{D}^{\alpha}\right)y(x)+q(x)y(x)=\lambda \omega(x)y(x),\ x\in[a,b],
\end{equation}
where $\alpha\in(0,1)$ or $\alpha\in(\frac12,1)$.
\item The sixth type of FSL formulation is recently introduced as follows \cite{MR3742689}:
 \begin{equation}\label{FSLF6}
 (c_1{}^{\pm}\mathcal{D}^{\alpha}+c_2{}^{\mp}\mathcal{D}^{\alpha})(p(x)(c_3{}^{\mp}\mathcal{D}^{\alpha}+c_4{}^{\pm}\mathcal{D}^{\alpha})y(x))=\lambda \omega(x)y(x),\ x\in[a,b].
 \end{equation}
We emphasize that both left  and right  fractional derivatives (${}^-\mathcal{D}^{\alpha}$ and ${}^+\mathcal{D}^{\alpha}$) appear in the above three types of the FSLPs.
\end{itemize}

 Other formulations of FSLPs can be founded in \cite{Khosravian-Arab2015526} and references therein.
 
Roughly speaking, from the numerical analysis view point,  only the last three formulations of FSLPs have some notable features. One of the most interesting features of these operators is that they produce some self-adjoint operators under suitable boundary conditions. Between the last three cases, operators \eqref{FSLF5} can be obtained from the well known Euler-Lagrange equations of a class of fractional variational problem  directly \cite{MR2351655,MR2332922} and therefore the researchers have interested to focus their studies on these operators. The first study on the operators \eqref{FSLF5} was made by Klimek and Agrawal in \cite{Klimek2013795}. They provide some spectral properties of these operators under suitable boundary conditions. They showed that:
\begin{itemize}
	\item Eigenvalues of these operators are all real and simple.
	\item Eigenfunctions of these operator corresponding to distinct  eigenvalues are orthogonal with respect to some suitable weight functions.
\end{itemize}

At a little time, Zayernouri and Karniadakis in \cite{Zayernouri2013495} gave the analytical solutions of FSLPs which were considered by Klimek and Agrawal. In fact, the authors derived the eigenfunctions of FSLPs of the type \eqref{FSLF5} and then some useful properties of these functions, which were  so-called {\it Jacobi poly-fractonomials}, have been introduced and studied in detail. The closed forms of Jacobi poly-fractonomials on $[0,T]$ are as follows \cite{Zayernouri2013495} (See also \cite{MR3654887,MR3654886,MR3371355}):
\begin{eqnarray}
{}^1\phi_n(x)&=& x^b P^{(a,b)}_n\left(\frac{2x}{T}-1\right),\label{Zayer_1}\\
{}^2\phi_n(x)&=& (T-x)^a P^{(a,b)}_n\left(\frac{2x}{T}-1\right).\label{Zayer_2}
\end{eqnarray}

Let us look  at formulas \eqref{Zayer_1} and \eqref{Zayer_2} more closely. First,  the function \eqref{Zayer_1} vanishes at $x=0, (b>0)$ and function \eqref{Zayer_2} vanishes at $x=T,\ (a>0)$. Second, in fact, these functions can be seen as the special cases of the M\"untz expansion \eqref{Eq_3} for $\lambda_k=b+k$ (and $\lambda_k=a+k$).

To the best knowledge of the authors, until now, the following categories of the special cases of M\"untz expansions \eqref{Eq_3} have been introduced and studied in detail:
\begin{itemize}
	\item Take $\lambda_k=k\alpha,\ k=0,1,\ldots,N$  \cite{MR2316514,MR3020667,MR3499043,MR3531734}.
	\item Take $\lambda_k=k\alpha+j,\ k,j\in\Bbb{N}_0,\ k\alpha+j<N$ \cite{MR3124341,MR3261539}.
	\item Take $\lambda_k=k+\alpha,\ k=0,1,\ldots,N$ \cite{Zayernouri2013495,Zayernouri2014312,Zayernouri2014A40,MR3471102}.
\end{itemize}

So, the main objective of this paper is to construct other special cases for the M\"untz expansions \eqref{Eq_3} which not only produce the spectral methods with exponential accuracy for the non-smooth problems (problems with singularities at the end points of their domains) but also have some interesting features such as: Orthogonality, Rodrigues' formula, three-term recurrence relation and etc. 

For the readers' convenience, we highlight the main contributions of this paper as follows: 

\begin{itemize} 
	\item We introduce two new classes of functions which we called Jacobi-M\"untz functions (MJFs-1 and MJFs-2) in the following forms:
	\begin{eqnarray*}
	\ \ \ \ \ \ \ \ 	{}^1\mathcal{J}^{(\alpha,\beta,\mu,\sigma,\eta)}_n(x)&&=x^{\sigma(\beta-\eta-\mu)}P_n^{(\alpha,\beta)}\left(2\left(\frac{x}{b}\right)^\sigma-1\right),\ \ x\in[0,b],
	\\
	{}^2\mathcal{J}^{(\alpha,\beta,\sigma,\eta)}_n(x)&&=x^{\sigma\eta}\left(b^\sigma-x^\sigma\right)^\alpha P_n^{(\alpha,\beta)}\left(2\left(\frac{x}{b}\right)^\sigma-1\right),\ x\in[0,b], 
	\end{eqnarray*}
	where $\sigma\geq0$. MJFs-1 and MJFs-2 are in fact generalization of all classes which were introduced before (See \cref{Special}).
	\item We show that they can be seen as the eigenfunctions of two self-adjoint Erd\'{e}lyi-Kober (EK) fractional Sturm-Liouville eigenvalue problems (See \cref{Thm_1}, \cref{Thm_2}, \cref{Thm_1-1}).
	\item We derive some interesting properties of JMFs-1 and JMFs-2 like as: Non-polynomial natures, ordinary derivatives, Rodrigues' formulas, three-term recurrence formulas, orthogonality and orthogonality of their EK fractional derivatives (See \cref{SomeProperties}).
	\item  We also show that they construct two complete sets in some suitable Hilbert spaces (See \cref{Complete}). 
	\item Two new orthogonal projections with respect to MJFs-1 and MJFs-2 are introduced and their error bounds are proved (See \cref{ErrorBounds}).
	\item Two new quadrature rules which we called Gauss-Jacobi-M\"untz quadrature rules are introduced (See  \cref{MintzJacQuad}).
	
	\item These new basis functions are also applied to solve some problems  like as: EK fractional steady-state diffusion equations,  linear EK fractional diffusion equations and nonlinear Burgers partial differential equations.
\end{itemize}

The outline of this paper is as follows. In the next section, we provide some necessary definitions and properties of Erd\'{e}lyi-Kober fractional derivatives (and integrals) and Jacobi polynomials. The main goal of this paper is introduced in \cref{sec:main}. In this section, we define two new Jacobi-M\"untz functions and derive their essential properties. Also the approximation results together with two new quadrature rules are included in this section. Some applications of the newly introduced basis functions are given in \cref{sec:experiments}, and the conclusions follow in
\cref{sec:conclusions}.

\section{Preliminaries}\label{sec:preliminaries}
\subsection{Fractional calculus}
In this section, we compile some basic definitions and properties of Erd\'{e}lyi-Kober
fractional  derivatives and integrals.
\begin{definition}\label{E-K-FI}
	The left and right Erd\'{e}lyi-Kober  fractional integrals ${}_{a}I_{x,\sigma,\eta}^{\mu}$ and ${}_{x}I_{b,\sigma,\eta}^{\mu}$ of order $\mu\in\Bbb{R}^+$  are defined by \cite{MR2218073}:
	\begin{equation}\label{RINTL}
{}_{a}I_{x,\sigma,\eta}^{\mu}[f](x)=\frac{\sigma x^{-\sigma(\eta+\mu)}}{\Gamma(\mu)}\int_a^x(x^\sigma-t^\sigma)^{\mu-1}t^{\sigma(\eta+1)-1}f(t)\,dt,\ x\in(a,b],\ a>0,
	\end{equation}
	and
	\begin{equation}\label{RINTR}
	{}_{x}I_{b,\sigma,\eta}^{\mu}[f](x)=\frac{\sigma x^{\sigma\eta}}{\Gamma(\mu)}\int_x^b(t^\sigma-x^\sigma)^{\mu-1}t^{-\sigma(\eta+\mu-1)-1}f(t)\,dt,\,\ x\in[a,b),\ a>0,
		\end{equation}
respectively. Here $\Gamma(.)$ denotes the Euler gamma function.
\end{definition}
\begin{remark}It is interesting to point out that \cref{E-K-FI} for $\mu=1$ reduces to the following integral formulas respectively \cite{MR2218073}:
		\begin{equation*}\label{RINTLSC}
		{}_{a}I_{x,\sigma,\eta}^{1}[f](x)=\sigma x^{-\sigma(\eta+1)}\int_a^xt^{\sigma(\eta+1)-1}f(t)\,dt,\ x\in(a,b],\ a>0,
		\end{equation*}
		\begin{equation*}\label{RINTRSC}
		{}_{x}I_{b,\sigma,\eta}^{1}[f](x)=\sigma x^{\sigma\eta}\int_x^bt^{-\sigma\eta-1}f(t)\,dt,\,\ x\in[a,b),\ a>0.
		\end{equation*}
\end{remark}
\begin{definition}\label{E-KFD}
	The left and right Erd\'{e}lyi-Kober  fractional derivatives ${}_{a}D_{x,\sigma,\eta}^{\mu}$ and ${}_{x}D_{b,\sigma,\eta}^{\mu}$ of order $0<\mu<1$  are defined by \cite{MR2218073}:
	\begin{equation}\label{RDEERL}
	{}_{a}D_{x,\sigma,\eta}^{\mu}[f](x)=x^{-\sigma\eta}\left(\frac{1}{\sigma x^{\sigma-1}}\frac{d}{dx}\right)x^{\sigma(\eta+1)}{}_{a}I_{x,\sigma,\eta+\mu}^{1-\mu}[f](x),\ x\in(a,b],
	\end{equation}
	and
	\begin{equation}\label{RDERR}
	{}_{x}D_{b,\sigma,\eta}^{\mu}[f](x)=x^{\sigma(\eta+\mu)}\left(\frac{-1}{\sigma x^{\sigma-1}}\frac{d}{dx}\right)x^{-\sigma(\eta+\mu-1)}{}_{x}I_{b,\sigma,\eta+\mu-1}^{1-\mu}[f](x),\ x\in[a,b),
	\end{equation}
	respectively.
\end{definition}
\begin{remark}\label{Spec-E-KFD}
It is worthwhile to point out that for $\mu=1$, the \cref{E-KFD} reduces to:
\begin{eqnarray*}
&&{}_{a}D_{x,\sigma,\eta}^{1}[f](x)=x^{-\sigma\eta}\left(\frac{1}{\sigma x^{\sigma-1}}\frac{d}{dx}\right)x^{\sigma(\eta+1)}f(x),\\
&&{}_{x}D_{b,\sigma,\eta}^{1}[f](x)=x^{\sigma(\eta+1)}\left(\frac{-1}{\sigma x^{\sigma-1}}\frac{d}{dx}\right)x^{-\sigma\eta}f(x).
\end{eqnarray*}
\end{remark}
\begin{definition}
	The left and right  Erd\'{e}lyi-Kober fractional derivatives of Caputo type ${}^{C}_{a}D_{x,\sigma,\eta}^{\mu}$ and ${}^{C}_{x}D_{b,\sigma,\eta}^{\mu}$ of order $0<\mu<1$  are defined by \cite{MR2218073}:
	\begin{equation}\label{CDEERL}
	{}^{C}_{a}D_{x,\sigma,\eta}^{\mu}[f](x)=\frac{ x^{-\sigma\eta}}{\Gamma(1-\mu)}\int_a^x(x^\sigma-t^\sigma)^{-\mu}\frac{d}{dt}\left(t^{\sigma(\eta+\mu)}f(t)\right)\,dt,\ x\in(a,b],\ a>0,
	\end{equation}
	and
	\begin{equation}\label{CDERR}
	{}^{C}_{x}D_{b,\sigma,\eta}^{\mu}[f](x)=-\frac{ x^{\sigma(\eta+\mu)}}{\Gamma(1-\mu)}\int_x^b(t^\sigma-x^\sigma)^{-\mu}\frac{d}{dt}\left(t^{-\sigma\eta}f(t)\right)\,dt,\ x\in[a,b),\ a>0,
	\end{equation}
	respectively.
\end{definition}
In the following, we present some useful properties of the Erd\'{e}lyi-Kober (EK) fractional integrals and derivatives.
\begin{property} (Fractional integration by parts). The following properties hold true \cite{MR2218073}:
	\begin{itemize}
		\item Let $\mu>0$. If the left and right  EK fractional integrals of the given functions $f$ and $g$ exist, then  we have:
		\begin{equation}
		\int_a^b x^{\sigma-1}g(x)\ {}_{a}I_{x,\sigma,\eta}^{\mu}[f](x)\,dx=	\int_a^b x^{\sigma-1}f(x)\ {}_{x}I_{b,\sigma,\eta}^{\mu}[g](x)\,dx,\ a>0.
		\end{equation}
			\item Let $0<\mu<1$. If the left and right  EK fractional derivatives of the given functions $f$ and $g$ exist, then for $a\geq0$ we have:
			\begin{eqnarray}
			\int_a^b x^{\sigma-1}g(x)\ {}_{a}D_{x,\sigma,\eta}^{\mu}[f](x)\,dx&=&	\int_a^b x^{\sigma-1}f(x)\ {}^C_{x}D_{b,\sigma,\eta}^{\mu}[g](x)\,dx\label{Frac_Int_1}\\
			&+&\Big[\frac{x^\sigma}{\sigma}g(x)\ {}_{a}I_{x,\sigma,\eta+\mu}^{1-\mu}[f](x)\Big]_{x=a}^{x=b},\nonumber\\
			\int_a^b x^{\sigma-1}g(x)\ {}_{x}D_{b,\sigma,\eta}^{\mu}[f](x)\,dx&=&	\int_a^b x^{\sigma-1}f(x)\ {}^C_{a}D_{x,\sigma,\eta}^{\mu}[g](x)\,dx\label{Frac_Int_2}\\
			&-&\Big[\frac{x^\sigma}{\sigma}g(x)\ {}_{x}I_{b,\sigma,\eta+\mu-1}^{1-\mu}[f](x)\Big]_{x=a}^{x=b}.\nonumber
			\end{eqnarray}
	\end{itemize}
	\begin{proof}
		The proofs are straightforward.
		\end{proof}
\end{property}	
\begin{property}
Let $0<a<b<\infty$ and $\alpha,\beta>0$. If $f(x)\in L^p(a,b)$, then we have:
\begin{equation}
{}_{a}I_{x,\sigma,\eta}^{\alpha}{}\ {}_{a}I_{x,\sigma,\eta+\alpha}^{\beta}[f](x)={}_{a}I_{x,\sigma,\eta}^{\alpha+\beta}[f](x).
\end{equation}
and
\begin{equation}
{}_{x}I_{b,\sigma,\eta}^{\alpha}{}\ {}_{x}I_{b,\sigma,\eta+\alpha}^{\beta}[f](x)={}_{x}I_{b,\sigma,\eta}^{\alpha+\beta}[f](x).
\end{equation}
\end{property}	
\begin{property}
Let $0<a<b<\infty$. Then for sufficiently good function $f(x)$, we have:
\begin{equation}
{}_{a}D_{x,\sigma,\eta}^{\alpha}{}\ {}_{a}I_{x,\sigma,\eta}^{\alpha}{}[f](x)=f(x),\ {}_{x}D_{b,\sigma,\eta}^{\alpha}{}\ {}_{x}I_{b,\sigma,\eta}^{\alpha}{}[f](x)=f(x).
\end{equation}	
\end{property} 
\subsection{Jacobi polynomials} In this subsection, we briefly review some properties of Jacobi polynomials \cite{MR2867779,MR3471102}.

The hypergeometric functions are defined by:
\begin{equation}\label{Hyper_1}
	 {}_2 F_1\left(\begin{matrix}a,\ b\\ c\end{matrix};x\right)=\sum_{j=0}^{\infty}\frac{(a)_j(b)_j}{(c)_j}\frac{x^j}{j!},\ |x|<1,\ a,b,c\in\Bbb{R},\ -c\notin\Bbb{N}_0,
\end{equation}
where $(\theta)_j$  stands for the Pochhammer symbol. For $\theta\in \Bbb{R},\ j\in\Bbb{N}_0$ we have:
\begin{equation}\label{PochSymb}
(\theta)_0=1;\ (\theta)_j:=\theta(\theta+1)(\theta+2)\ldots(\theta+j-1)=\frac{\Gamma(\theta+j)}{\Gamma(\theta)},\ \text{for}\  j\geq1.
\end{equation}

For negative integer number $a$ or $b$, the hypergeometric functions \eqref{Hyper_1} reduces to a polynomial. 

The Jacobi polynomials with parameters $\alpha,\beta\in\Bbb{R}$ are obtained by the following formulas:
\begin{eqnarray}
&& P_n^{(\alpha,\beta)}(x)=\frac{(\alpha+1)_n}{n!}\
 {}_2 F_1\left(\begin{matrix}-n,\ n+\alpha+\beta+1\\ \alpha+1\end{matrix};\frac{1-x}{2}\right),\ n\geq1,\label{Hyper_Jac_1}\\
&& P_n^{(\alpha,\beta)}(x)=(-1)^n\frac{(\beta+1)_n}{n!}\ {}_2 F_1\left(\begin{matrix}-n,\ n+\alpha+\beta+1\\ \beta+1\end{matrix};\frac{1+x}{2}\right), \ n\geq1,\label{Hyper_Jac_2}
\end{eqnarray}

The Jacobi polynomials with $\alpha,\beta\in\Bbb{R}$, satisfy in the following three-term recurrence relation:
\begin{equation}
A_n^{\alpha,\beta}P_{n+1}^{(\alpha,\beta)}(x)=\left(B_n^{\alpha,\beta}x-C_n^{\alpha,\beta}\right)P_{n}^{(\alpha,\beta)}(x)-E_n^{\alpha,\beta}P_{n-1}^{(\alpha,\beta)}(x),\ n\geq1,
\end{equation}
where 
\begin{equation*}
	P_{0}^{(\alpha,\beta)}(x)=1,\ P_{1}^{(\alpha,\beta)}(x)=\frac{1}{2}(\alpha+\beta+2)x+\frac{1}{2}(\alpha-\beta),\ \alpha,\beta\in \Bbb{R},
\end{equation*} 
and by noting that:
\begin{eqnarray}
&& A_n^{\alpha,\beta}=2(n+1)(n+\alpha+\beta+1)(2n+\alpha+\beta),\label{Eq_Rec_1}\\
&& B_n^{\alpha,\beta}=(2n+\alpha+\beta)(2n+\alpha+\beta+1)(2n+\alpha+\beta+2),\label{Eq_Rec_2}\\
&& C_n^{\alpha,\beta}=(\beta^2-\alpha^2)(2n+\alpha+\beta+1),\label{Eq_Rec_3}\\
&& E_n^{\alpha,\beta}=2(n+\alpha)(n+\beta)(2n+\alpha+\beta+1).\label{Eq_Rec_4}
\end{eqnarray}
For $\alpha,\beta>-1$, the classical Jacobi polynomials are orthogonal over $[-1,1]$ with respect to the weight function: $w^{(\alpha,\beta)}(x)=(1-x)^\alpha(1+x)^\beta$, i.e.,
\begin{equation}\label{Orthog}
\int_{-1}^{1}P_{n}^{(\alpha,\beta)}(x)P_{m}^{(\alpha,\beta)}(x)w^{(\alpha,\beta)}(x)\,dx=\gamma_n^{(\alpha,\beta)}\delta_{nm},
\end{equation}
where $\delta_{nm}$ stands for the Dirac Delta symbol and we also have: 
\begin{equation}\label{Orthog_Cons}
\gamma_n^{(\alpha,\beta)}=\frac{2^{\alpha+\beta+1}\Gamma(n+\alpha+1)\Gamma(n+\beta+1)}{(2n+\alpha+\beta+1)n!\Gamma(n+\alpha+\beta+1)}.
\end{equation}
The Rodrigues' formula for the Jacobi polynomials is as follows:
\begin{equation}
P_{n}^{(\alpha,\beta)}(x)=(1-x)^{-\alpha}(1+x)^{-\beta}\frac{(-1)^n}{2^nn!}\frac{d^n}{dx^n}\Big[(1-x)^{n+\alpha}(1+x)^{n+\beta}\Big].
\end{equation}

Finally, we recall the Bateman fractional integral formula \cite{MR1688958}. For $c,\mu\geq0$ and $|x|<1$, we have:
\begin{equation}\label{BateFracInt}
\frac{1}{\Gamma(\mu)}\int_{0}^{x}(x-t)^{\mu-1}t^{c-1}{}_2 F_1\left(\begin{matrix}a,\ b\\ c\end{matrix};t\right)\,dt=\frac{\Gamma(\mu)}{\Gamma(c+\mu)}x^{c+\mu-1}{}_2 F_1\left(\begin{matrix}a,\ b\\ c+\mu\end{matrix};x\right).
\end{equation}

After briefly reviewing the basic properties of the E-K fractional integrals and derivatives and the Jacobi polynomials, in this position, we are ready  to state the main aim of this paper in the next section.

\section{Main results}\label{sec:main}
\subsection{Jacobi-M\"untz functions}
In this section, we first introduce two subclasses of the M\"untz functions and in the sequel several interesting properties of them will be addressed.  
	\begin{definition}\label{JacMunFuns}
		Let $\alpha,\beta>-1$. The Jacobi-M\"unts functions of the first and second kinds (JMFs-1 and JMFs-2) are denoted by ${}^1\mathcal{J}^{(\alpha,\beta,\mu,\sigma,\eta)}_n(x)$ and ${}^2\mathcal{J}^{(\alpha,\beta,\sigma,\eta)}_n(x)$, respectively, and are defined by:
		\begin{eqnarray}
		\ \ \ \ \ \ \ \ 	{}^1\mathcal{J}^{(\alpha,\beta,\mu,\sigma,\eta)}_n(x)&&=x^{\sigma(\beta-\eta-\mu)}P_n^{(\alpha,\beta)}\left(2\left(\frac{x}{b}\right)^\sigma-1\right),\ \ x\in[0,b],
		\\
		{}^2\mathcal{J}^{(\alpha,\beta,\sigma,\eta)}_n(x)&&=x^{\sigma\eta}\left(b^\sigma-x^\sigma\right)^\alpha P_n^{(\alpha,\beta)}\left(2\left(\frac{x}{b}\right)^\sigma-1\right),\ x\in[0,b], 
		\end{eqnarray}
		where $\sigma\geq0$.
	\end{definition}
		\begin{remark}
			It should be noted that the JMFs-1 and JMFs-2 are in fact two new subclasses of M\"untz functions because for JMFs-1 we have:  
			\begin{equation*}
			{}^1\mathcal{J}^{(\alpha,\beta,\mu,\sigma,\eta)}_n(x)\in\text{span}\left\{x^{\lambda_k},k=0,1,\ldots,n\right\},
			\end{equation*}
			where $\lambda_k=a+kb,\ k=0,1,\ldots,n$ and  $a=\sigma(\beta-\eta-\mu),\ b=\sigma$. Moreover, for JMFs-2 we also have: 
			\begin{equation*}
			{}^2\mathcal{J}^{(\alpha,\beta,\sigma,\eta)}_n(x)\in\text{span}\left\{(b^\sigma-x^\sigma)^\alpha x^{\lambda_k},k=0,1,\ldots,n\right\},
			\end{equation*}
			where $\lambda_k=\sigma\eta+\sigma k$.
		\end{remark}
		\begin{remark}\label{Special}
			Another important issue which is worth to emphasize here is that the JMFs-1 and JMFs-2 are in fact two generalized classes of all modifications of functions constructed from Jacobi polynomials. In the following we list some of them:
			\begin{itemize}
				\item If $\sigma=1$ and $\beta=\eta+\mu$ then JMFs-1 reduces to the classical Jacobi polynomials on $[0,b]$.
				\item If $\sigma=1$ then the JMFs-1 reduces to the first type of the Jacobi poly-fractonomials on $[0,b]$ \cite{Zayernouri2013495,MR2787811} (See \eqref{Zayer_1}).
				\item If $\beta=\eta+\mu$ then the JMFs-1 reduces to fractional Jacobi (and in a special case fractional Legendre) functions \cite{MR2316514,MR3020667,MR3499043,MR3531734}. 
				\item If $\sigma=1$ and $\eta=0$ then the JMFs-2 reduces to the second type of the Jacobi poly-fractonomials on $[0,b]$ \cite{Zayernouri2013495} (See \eqref{Zayer_2}).
				\item If $\sigma=1$ and $\eta=\beta$ the JMFs-2 reduces to the two-sided Jacobi poly-fractonomials on $[0,b]$ \cite{MR3742689}.
			\end{itemize}  
		\end{remark}
\subsection{Jacobi-M\"untz fractional Sturm-Liouville problems}		
In the following, we first introduce two new classes of FSL operators and then some important properties of them will be studied in detail. To do this, we denote:
\begin{equation*}
L^2((a,b),w(x)):=\left\{f:\ \int_{a}^{b}|f(x)|^2w(x)\,dx<\infty \right\}.
\end{equation*}

For $\mu\in(0.5,1)$, consider the following operators:
\begin{eqnarray}
&&{}^1\mathcal{L}^{(\alpha,\beta,\mu,\sigma,\eta)}(u(x)):=\frac{x^{\sigma-1}}{w_1^{(\alpha,\beta,\mu,\sigma,\eta)}(x)}\left({}_{x}D_{b,\sigma,\eta}^{\mu}\left[p_1(x)\ {}^{C}_{0}D_{x,\sigma,\eta}^{\mu}\right]u\right),\ \label{Oper_1}\\ &&{}^2\mathcal{L}^{(\alpha,\beta,\mu,\sigma,\eta)}(u(x)):=\frac{x^{\sigma-1}}{w_2^{(\alpha,\beta,\sigma,\eta)}(x)}\left({}_{0}D_{x,\sigma,\eta}^{\mu}\left[p_2(x)\ {}^{C}_{x}D_{b,\sigma,\eta}^{\mu}\right]u\right),\label{Oper_2}\ 
\end{eqnarray}
in $L^2((0,b),x^{\sigma-1}w_1^{(\alpha,\beta,\mu,\sigma,\eta)}(x))$ and $L^2((0,b),x^{\sigma-1}w_2^{(\alpha,\beta,\sigma,\eta)}(x))$, respectively, where
\begin{eqnarray}
&&  w_1^{(\alpha,\beta,\mu,\sigma,\eta)}(x)=x^{\sigma(2(\eta+\mu)-\beta)}\left(b^\sigma-x^\sigma\right)^{\alpha},\label{Wei_1}\\
&& p_1(x)=x^{\sigma(2\eta+\mu-\beta)}\left(b^\sigma-x^\sigma\right)^{\mu+\alpha}, \label{Coef_1}\\
&& w_2^{(\alpha,\beta,\sigma,\eta)}(x)=x^{\sigma(\beta-2\eta)}\left(b^\sigma-x^\sigma\right)^{-\alpha},\label{Wei_2}\\
&& p_2(x)=x^{-\sigma(\mu+2\eta-\beta)}\left(b^\sigma-x^\sigma\right)^{\mu-\alpha}. \label{Coef_2}
\end{eqnarray}

Now, we start to study some properties of the proposed operators \eqref{Oper_1} and \eqref{Oper_2}. 

\begin{theorem}\label{Thm_1}
The operators \eqref{Oper_1} and \eqref{Oper_2}, together with the functions \eqref{Wei_1}-\eqref{Coef_2} are both self-adjoint on the domains 
\begin{eqnarray}
&&\mathbf D_1:=\Big\{u\in L^2((0,b),x^{\sigma-1}w_1^{(\alpha,\beta,\mu,\sigma,\eta)}(x)),\nonumber\\ &&\hspace{3cm}\Big[x^\sigma u(x)\Big]_{x=0}=0,\ \Big[{}_{x}I_{b,\sigma,\eta+\mu-1}^{1-\mu}\left[p_1(x)\ {}^{C}_{0}D_{x,\sigma,\eta}^{\mu}u(x)\right]\Big]_{x=b}=0 \Big\},\nonumber\\
&& \mathbf D_2:=\Big\{u\in L^2((0,b),x^{\sigma-1}w_2^{(\alpha,\beta,\sigma,\eta)}(x)),\nonumber\\ &&\hspace{3cm}u(b)=0,\ \Big[x^\sigma {}_{0}I_{x,\sigma,\eta+\mu}^{1-\mu}\left[p_2(x)\ {}^{C}_{x}D_{b,\sigma,\eta}^{\mu}u(x)\right]\Big]_{x=0}=0 \Big\},\nonumber
\end{eqnarray}
respectively, that is for all $u,v\in\mathbf D_i,\ i=1,2$, we have:
\begin{equation*}
 \left({}^1\mathcal{L}^{(\alpha,\beta,\mu,\sigma,\eta)}u,v\right)_{w_1^{(\alpha,\beta,\mu,\sigma,\eta)}}=\left(u,{}^1\mathcal{L}^{(\alpha,\beta,\mu,\sigma,\eta)}v\right)_{w_1^{(\alpha,\beta,\mu,\sigma,\eta)}},
\end{equation*}
and
\begin{equation*}
	\left({}^2\mathcal{L}^{(\alpha,\beta,\mu,\sigma,\eta)}u,v\right)_{w_2^{(\alpha,\beta,\sigma,\eta)}}=\left(u,{}^2\mathcal{L}^{(\alpha,\beta,\sigma,\eta)}v\right)_{w_2^{(\alpha,\beta,\sigma,\eta)}},
\end{equation*}
where
\begin{equation*}
(f,g)_\omega=\int_0^b{\overline{f(z)}}g(z)\omega(z)\,dz.
\end{equation*}
\begin{proof} We prove this theorem only for $i=1$. Using fractional integration by parts \eqref{Frac_Int_1} and \eqref{Frac_Int_2} for all real valued functions $u,v\in\mathbf D_1$, we obtain:
	\begin{eqnarray*}
	\left({}^1\mathcal{L}^{(\alpha,\beta,\mu,\sigma,\eta)}u,v\right)_{w_1^{(\alpha,\beta,\mu,\sigma,\eta)}}&=&\int_0^bx^{\sigma-1}\left({}_{x}D_{b,\sigma,\eta}^{\mu}\left[p_1(x)\ {}^{C}_{0}D_{x,\sigma,\eta}^{\mu}u(x)\right]\right)v(x)\,dx\nonumber\\
	&=&\left(u,{}^1\mathcal{L}^{(\alpha,\beta,\mu,\sigma,\eta)}v\right)_{w_1^{(\alpha,\beta,\mu,\sigma,\eta)}}\\
	&+&\Big[\frac{x^\sigma}{\sigma}u(x){}_{x}I_{b,\sigma,\eta+\mu-1}^{1-\mu}\left[p_1(x)\ {}^{C}_{0}D_{x,\sigma,\eta}^{\mu}v(x)\right]\Big]_{x=0}^{x=b}\\
	&-&\Big[\frac{x^\sigma}{\sigma}v(x){}_{x}I_{b,\sigma,\eta+\mu-1}^{1-\mu}\left[p_1(x)\ {}^{C}_{0}D_{x,\sigma,\eta}^{\mu}u(x)\right]\Big]_{x=0}^{x=b}.
	\end{eqnarray*}
	Thanks to the fact that $u,v\in\mathbf D_1$, the last two terms of the above equation vanish and then the proof is concluded.
\end{proof}
\end{theorem}
\begin{theorem}\label{Thm_1-1} The eigenvalues $ {}^i\Lambda^{(\alpha,\beta,\mu)}_n,\ i=1,2$ of the eigenvalue problems 
\begin{eqnarray}\label{Operat}
{}^i\mathcal{L}^{(\alpha,\beta,\mu,\sigma,\eta)}\left(u(x)\right)=x^{\sigma-1}\ {}^i\Lambda^{(\alpha,\beta,\mu)}_n u(x),\ i=1,2,
\end{eqnarray}
are all real and the  eigenfunctions (corresponding to distinct eigenvalues) are mutually orthogonal with respect to the weight functions $x^{\sigma-1}\ w_1^{(\alpha,\beta,\mu,\sigma,\eta)}(x),\ x^{\sigma-1}\ w_2^{(\alpha,\beta,\sigma,\eta)}(x)$, respectively. 
\end{theorem}
\begin{proof}
The proof of the theorem is easily concluded from the self-adjointness of  the operators ${}^i\mathcal{L}^{(\alpha,\beta,\mu,\sigma,\eta)}\left(u(x)\right),\ i=1,2$.	
\end{proof}

In this position, our aim is to obtain the analytical solutions of the operators \eqref{Operat}. To reach our aim, we essentially need  some important remarks.
\begin{remark}\label{E-KFIHyp}	Let $\mu>0$. Then we have:
	\begin{eqnarray*}
		&&{}_{0}I_{x,\sigma,\eta}^{\mu}\Big[x^{\sigma(c-\eta-1)}{}_2 F_1\left(\begin{matrix}a,\ b\\ c\end{matrix};\left(\frac{x}{b}\right)^\sigma\right)\Big]=\frac{\Gamma(c)}{\Gamma(c+\mu)}x^{\sigma(c-\eta-1)}{}_2 F_1\left(\begin{matrix}a,\ b\\ c+\mu\end{matrix};\left(\frac{x}{b}\right)^\sigma\right).\\
		&&{}_{x}I_{b,\sigma,\eta}^{\mu}\Big[x^{\sigma(\eta+\mu)}\left(b^\sigma-x^\sigma\right)^{c-1} {}_2 F_1\left(\begin{matrix}a,\ b\\ c\end{matrix};1-\left(\frac{x}{b}\right)^\sigma\right)\Big]=\\ &&\hspace{3.5cm}\frac{\Gamma(c)}{\Gamma(c+\mu)}x^{\sigma\eta}\left(b^\sigma-x^\sigma\right)^{c+\mu-1}{}_2 F_1\left(\begin{matrix}a,\ b\\ c+\mu\end{matrix};1-\left(\frac{x}{b}\right)^\sigma\right).
	\end{eqnarray*}
\end{remark}
\begin{proof}
	The use of the change of variable in Bateman fractional integral formula \eqref{BateFracInt}, the above relations can be concluded.
\end{proof}
\begin{remark}\label{E-KFDHyp}	Let $0<\mu\leq1$. Then we have:
	\begin{eqnarray*}
		&&{}_{0}D_{x,\sigma,\eta}^{\mu}\Big[x^{\sigma(c-\mu-\eta-1)}{}_2 F_1\left(\begin{matrix}a,\ b\\ c\end{matrix};\left(\frac{x}{b}\right)^\sigma\right)\Big]=\frac{\Gamma(c)}{\Gamma(c-\mu)}x^{\sigma(c-\mu-\eta-1)}{}_2 F_1\left(\begin{matrix}a,\ b\\ c-\mu\end{matrix};\left(\frac{x}{b}\right)^\sigma\right).\\
		&&{}_{x}D_{b,\sigma,\eta}^{\mu}\Big[x^{\sigma\eta}\left(b^\sigma-x^\sigma\right)^{c-1} {}_2 F_1\left(\begin{matrix}a,\ b\\ c\end{matrix};1-\left(\frac{x}{b}\right)^\sigma\right)\Big]=\\ &&\hspace{3.5cm}\frac{\Gamma(c)}{\Gamma(c-\mu)}x^{\sigma(\eta+\mu)}\left(b^\sigma-x^\sigma\right)^{c-\mu-1}{}_2 F_1\left(\begin{matrix}a,\ b\\ c-\mu\end{matrix};1-\left(\frac{x}{b}\right)^\sigma\right).
	\end{eqnarray*}
\end{remark}
\begin{proof}
	The proof is concluded by the use of \cref{E-KFD} and \cref{E-KFIHyp}.
\end{proof}
\begin{remark}\label{FracInt}
	Let $\mu>0$. Then we have:
	\begin{eqnarray*}
	&&{}_{0}I_{x,\sigma,\eta}^{\mu}\Big[x^{\sigma(\beta-\eta)}P_k^{(\alpha,\beta)}\left(2\left(\frac{x}{b}\right)^\sigma-1\right)\Big]=\\
	&&\hspace{5cm}\frac{\Gamma(k+\beta+1)}{\Gamma(k+\beta+\mu+1)}x^{\sigma(\beta-\eta)}P_k^{(\alpha-\mu,\beta+\mu)}\left(2\left(\frac{x}{b}\right)^\sigma-1\right).\\
	&&{}_{x}I_{b,\sigma,\eta}^{\mu}\Big[x^{\sigma(\eta+\mu)}\left(b^\sigma-x^\sigma\right)^\alpha P_k^{(\alpha,\beta)}\left(2\left(\frac{x}{b}\right)^\sigma-1\right)\Big]=\\ &&\hspace{3.5cm}\frac{\Gamma(k+\alpha+1)}{\Gamma(k+\alpha+\mu+1)}x^{\sigma\eta}\left(b^\sigma-x^\sigma\right)^{\alpha+\mu}P_k^{(\alpha+\mu,\beta-\mu)}\left(2\left(\frac{x}{b}\right)^\sigma-1\right).
	\end{eqnarray*}
\end{remark}
\begin{proof} 
	The proof of the first relation is easily obtain when we set $a=-k$, $b=k+\alpha+\beta+1$ and $c=\beta+1$ in \cref{E-KFIHyp} together with the use of formula \eqref{Hyper_Jac_2}. For the proof of the second relation we substitute $a=-k$, $b=k+\alpha+\beta+1$ and $c=\alpha+1$ in \cref{E-KFIHyp} and then using the relation \eqref{Hyper_Jac_1}.
	\end{proof}
\begin{remark}\label{FracDer}
	Let $0<\mu\leq1$. Then we have:
	\begin{eqnarray*}
		&&{}_{0}D_{x,\sigma,\eta}^{\mu}\Big[x^{\sigma(\beta-\eta-\mu)}P_k^{(\alpha,\beta)}\left(2\left(\frac{x}{b}\right)^\sigma-1\right)\Big]=\\
		&&\hspace{5cm}\frac{\Gamma(k+\beta+1)}{\Gamma(k+\beta-\mu+1)}x^{\sigma(\beta-\eta-\mu)}P_k^{(\alpha+\mu,\beta-\mu)}\left(2\left(\frac{x}{b}\right)^\sigma-1\right).\\
		&&{}_{x}D_{b,\sigma,\eta}^{\mu}\Big[x^{\sigma\eta}\left(b^\sigma-x^\sigma\right)^\alpha P_k^{(\alpha,\beta)}\left(2\left(\frac{x}{b}\right)^\sigma-1\right)\Big]=\\ &&\hspace{3.3cm}\frac{\Gamma(k+\alpha+1)}{\Gamma(k+\alpha-\mu+1)}x^{\sigma(\eta+\mu)}\left(b^\sigma-x^\sigma\right)^{\alpha-\mu}P_k^{(\alpha-\mu,\beta+\mu)}\left(2\left(\frac{x}{b}\right)^\sigma-1\right).
	\end{eqnarray*}
\end{remark}
\begin{proof}
	The proof of this theorem is fairly similar to the proof of the previous remark.
\end{proof}
\begin{theorem}\label{Thm_2}
	The analytical solutions of the eigenvalue problems 
	\begin{eqnarray}\label{EigenValue_1}
	{}^1\mathcal{L}^{(\alpha,\beta,\mu,\sigma,\eta)}\left(u_1(x)\right)=x^{\sigma-1}\ {}^1\Lambda^{(\alpha,\beta,\mu)}_n u_1(x),\ \alpha>\mu-1,\ \beta>\eta+\mu,\ 
	\end{eqnarray}
	subject to the boundary conditions:
	\begin{equation}\label{IBCEigenValue_1}
	 u_1(0)=0,\ {}_{x}I_{b,\sigma,\eta+\mu-1}^{1-\mu}\left[p_1(x)\ {}^{C}_{0}D_{x,\sigma,\eta}^{\mu}u_1(x)\right]\Big]_{x=b}=0,
	\end{equation}
	and 
	\begin{eqnarray}\label{EigenValue_2}
	{}^2\mathcal{L}^{(\alpha,\beta,\mu,\sigma,\eta)}\left(u_2(x)\right)=x^{\sigma-1}\  {}^2\Lambda^{(\alpha,\beta,\mu)}_n u_2(x),\ \alpha>0,\ \beta>\eta-\mu+1,
	\end{eqnarray}
	subject to the boundary conditions:
	\begin{equation}\label{IBCEigenValue_2}
	\ {}_{0}I_{x,\sigma,\eta+\mu}^{1-\mu}\left[p_2(x)\ {}^{C}_{x}D_{b,\sigma,\eta}^{\mu}u_2(x)\right]\Big]_{x=0}=0, \ u_2(b)=0,
	\end{equation}
	where the operators ${}^i\mathcal{L}^{(\alpha,\beta,\mu,\sigma,\eta)}(.),\ i=1,2$ are defined in  \eqref{Oper_1} and \eqref{Oper_2}, together with the functions \eqref{Wei_1}-\eqref{Coef_2} and
	\begin{eqnarray}
&&	{}^1\Lambda^{(\alpha,\beta,\mu)}_n=\frac{\Gamma(n+\beta+1)\Gamma(n+\alpha+\mu+1)}{\Gamma(n+\beta-\mu+1)\Gamma(n+\alpha+1)},\\ &&{}^2\Lambda^{(\alpha,\beta,\mu)}_n=\frac{\Gamma(n+\alpha+1)\Gamma(n+\beta+\mu+1)}{\Gamma(n+\alpha-\mu+1)\Gamma(n+\beta+1)},
	\end{eqnarray}
	are
	 $${}^1\mathcal{J}^{(\alpha,\beta,\mu,\sigma,\eta)}_n(x)=x^{\sigma(\beta-\eta-\mu)}P_n^{(\alpha,\beta)}\left(2\left(\frac{x}{b}\right)^\sigma-1\right),\  \alpha>\mu-1,\ \beta>\eta+\mu,$$
	  and $${}^2\mathcal{J}^{(\alpha,\beta,\sigma,\eta)}_n(x)=x^{\sigma\eta}\left(b^\sigma-x^\sigma\right)^\alpha P_n^{(\alpha,\beta)}\left(2\left(\frac{x}{b}\right)^\sigma-1\right),\ \alpha>0,\ \beta>\eta-\mu+1,$$ respectively.
\end{theorem}
\begin{proof}
	We only obtain the analytical solution to the eigenvalue problem  \eqref{EigenValue_1}. The same fashion can be used for the eigenvalue problem \eqref{EigenValue_2}.  First, we recall the following relations which will be obtained easily from \cref{FracDer} as follows:
		\begin{eqnarray*}
			&&{}_{x}D_{b,\sigma,\eta}^{\mu}\Big[x^{\sigma\eta}\left(b^\sigma-x^\sigma\right)^{\alpha+\mu} P_n^{(\alpha+\mu,\beta-\mu)}\left(2\left(\frac{x}{b}\right)^\sigma-1\right)\Big]=\\ &&\hspace{3.3cm}\frac{\Gamma(n+\alpha+\mu+1)}{\Gamma(n+\alpha+1)}x^{\sigma(\eta+\mu)}\left(b^\sigma-x^\sigma\right)^{\alpha}P_n^{(\alpha,\beta)}\left(2\left(\frac{x}{b}\right)^\sigma-1\right).
		\end{eqnarray*}
Then, the use of \cref{FracDer} together with the above relation, we have:
\begin{eqnarray}
	&&\frac{x^{\sigma-1}}{w_1^{(\alpha,\beta,\mu,\sigma,\eta)}(x)}\left({}_{x}D_{b,\sigma,\eta}^{\mu}\left[p_1(x)\ {}^{C}_{0}D_{x,\sigma,\eta}^{\mu}{}^1\mathcal{J}^{(\alpha,\beta,\mu,\sigma,\eta)}_n(x)\right]\right)=\nonumber\\
	&&\frac{\Gamma(n+\beta+1)}{\Gamma(n+\beta-\mu+1)}\frac{x^{\sigma-1}}{w_1^{(\alpha,\beta,\mu,\sigma,\eta)}(x)}	
	\left({}_{x}D_{b,\sigma,\eta}^{\mu}\left[x^{\sigma\eta}\left(b^\sigma-x^\sigma\right)^{\alpha+\mu} P_n^{(\alpha+\mu,\beta-\mu)}\left(2\left(\frac{x}{b}\right)^\sigma-1\right)\right]\right)\nonumber\\
	&&\hspace{7.5cm}=x^{\sigma-1}\  {}^1\Lambda^{(\alpha,\beta,\mu,\sigma,\eta)}_n\ {}^1\mathcal{J}^{(\alpha,\beta,\mu,\sigma,\eta)}_n(x),\nonumber
\end{eqnarray}
which means that the function ${}^1\mathcal{J}^{(\alpha,\beta,\mu,\sigma,\eta)}_n(x)$ satisfies in eigenvalue problem \eqref{EigenValue_1}. Now, it remains to show that the function ${}^1\mathcal{J}^{(\alpha,\beta,\mu,\sigma,\eta)}_n(x)$ is also satisfied the boundary conditions \eqref{IBCEigenValue_1}. The first condition holds true, that is ${}^1\mathcal{J}^{(\alpha,\beta,\mu,\sigma,\eta)}_n(0)=0$ for $\beta>\mu+\eta$. For the second condition, we have:
\begin{eqnarray*}
&&{}_{x}I_{b,\sigma,\eta+\mu}^{1-\mu}\left[p_1(x)\ {}^{C}_{0}D_{x,\sigma,\eta}^{\mu}{}^1\mathcal{J}^{(\alpha,\beta,\mu,\sigma,\eta)}_n(x)\right]=\\
&&\ \ \ \frac{\Gamma(k+\beta+1)}{\Gamma(k+\beta-\mu+1)} {}_{x}I_{b,\sigma,\eta+\mu}^{1-\mu}\Big[x^{\sigma\eta}\left(b^\sigma-x^\sigma\right)^{\alpha+\mu} P_k^{(\alpha+\mu,\beta-\mu)}\left(2\left(\frac{x}{b}\right)^\sigma-1\right)\Big]=\nonumber\\
&&\ \ \ \frac{\Gamma(k+\beta+1)\Gamma(k+\alpha+1)}{\Gamma(k+\beta-\mu+1)\Gamma(k+\alpha-\mu+2)}x^{\sigma(\eta+\mu-1)}(b^\sigma-x^\sigma)^{\alpha-\mu+1}P_k^{(\alpha-\mu+1,\beta+\mu-1)}\left(2\left(\frac{x}{b}\right)^\sigma-1\right),
\end{eqnarray*}
which tends to zero at $x=b$ for $\alpha>\mu-1$. This completes the proof.
	\end{proof}

		\begin{remark}\label{EKFD}
			By noting \cref{JacMunFuns}, we can rewrite \cref{FracDer} for $0<\mu\leq1$  follows:
			\begin{eqnarray*}
				&&{}_{0}D_{x,\sigma,\eta}^{\mu}\Big[{}^1\mathcal{J}^{(\alpha,\beta,\mu,\sigma,\eta)}_k(x)\Big]=\frac{\Gamma(k+\beta+1)}{\Gamma(k+\beta-\mu+1)}{}^1\mathcal{J}^{(\alpha+\mu,\beta-\mu,\mu,\sigma,\eta-\mu)}_k(x),\  \beta-\mu>-1.\\
				&&{}_{x}D_{b,\sigma,\eta}^{\mu}\Big[{}^2\mathcal{J}^{(\alpha,\beta,\sigma,\eta)}_k(x)\Big]=\frac{\Gamma(k+\alpha+1)}{\Gamma(k+\alpha-\mu+1)}{}^2\mathcal{J}^{(\alpha-\mu,\beta+\mu,\sigma,\eta+\mu)}_k(x),\  \alpha-\mu>-1.
			\end{eqnarray*}
			\end{remark}
\begin{remark}\label{ExtendedEKFD}
	It is easy to show that \cref{EKFD} remains true for the case $ \mu>1$.
\end{remark}		
	\subsection{Some properties of MJFs-1 and MJFs-2}\label{SomeProperties}
In what follows we list some properties of 	the Jacobi-M\"unts functions of the first and second kinds as the eigenfunctions of the FSLPs \eqref{EigenValue_1} and \eqref{EigenValue_2}.
\begin{itemize}
	\item {\bf Non-polynomial natures:}\\
	Due to the shapes of JMFs-1 and JMFs-2, we clearly find that these eigenfunctions have non-polynomial behaviors. So, we can conclude that these functions are suitable to apply for the problems with non-smooth solutions. 
	
	Also, for $\beta>\mu+\eta$, we have  ${}^1\mathcal{J}^{(\alpha,\beta,\mu,\sigma,\eta)}_n(0)=0$ and for $\eta,\alpha>0$ we have ${}^2\mathcal{J}^{(\alpha,\beta,\sigma,\eta)}_n(0)={}^2\mathcal{J}^{(\alpha,\beta,\sigma,\eta)}_n(b)=0$.
	\item {\bf  Asymptotic behavior of eigenvalues ${}^1\Lambda^{(\alpha,\beta,\mu)}_n$ and ${}^2\Lambda^{(\alpha,\beta,\mu)}_n$:}\\
	Using the well known formula \cite{MR1225604}:
	 \begin{equation}\label{Asymp1}
	 \frac{\Gamma(n+a)}{\Gamma(n+b)}\sim n^{a-b}\left(1+\frac{(a-b)(a+b-1)}{2n}+\mathcal{O}\left(\frac{1}{n^2}\right)+\cdots\right),
	 \end{equation}
	 as $n\rightarrow\infty$, providing $n\neq-a,\ -a-1,\ldots$  and $n\neq-b,\ -b-1,\ldots$. Hence, the asymptotic behavior of the eigenvalues ${}^i\Lambda_n^{(\alpha,\beta,\mu)},\ i=1,2$ for $\mu\in(0,1)$ is for $ n\gg1$ as follows:
	 \[
{}^1\Lambda^{(\alpha,\beta,\mu)}_n=\frac{\Gamma(n+\beta+1)\Gamma(n+\alpha+\mu+1)}{\Gamma(n+\beta-\mu+1)\Gamma(n+\alpha+1)}\sim n^{2\mu}=\left\{
	 \begin{array}{ll}
	 n^2, & \hbox{$\mu\rightarrow1$;} \\
	 {n}, & \hbox{$\mu\rightarrow\frac{1}{2}$;} \\
	 1, & \hbox{$\mu\rightarrow0$.}
	 \end{array}
	 \right.
	 ,
	 \]
	 \[
{}^2\Lambda^{(\alpha,\beta,\mu)}_n=\frac{\Gamma(n+\alpha+1)\Gamma(n+\beta+\mu+1)}{\Gamma(n+\alpha-\mu+1)\Gamma(n+\beta+1)}\sim n^{2\mu}=\left\{
	 \begin{array}{ll}
	 n^2, & \hbox{$\mu\rightarrow1$;} \\
	 {n}, & \hbox{$\mu\rightarrow\frac{1}{2}$;} \\
	 1, & \hbox{$\mu\rightarrow0$.}
	 \end{array}
	 \right..
	 \]
So, these results indicate that the eigenvalues of FSLPs 	\eqref{EigenValue_1} and \eqref{EigenValue_2} grow sub-quadratically as $n\rightarrow\infty$. We also point out that these results for $\mu=1$ coincide with the growth of the eigenvalues of the classical Jacobi Sturm-Liouville problem.
	\item {\bf The first derivative:}\\
	For two special cases of JMFs-1 and JMFs-2, we have:
	\begin{eqnarray*}
	&&\frac{d}{dx}\left[x^{\sigma\beta}P_n^{(\alpha,\beta)}\left(2\left(\frac{x}{b}\right)^\sigma-1\right)\right]=\nonumber\\
	&&\hspace{3cm}\frac{\sigma\Gamma(n+\beta+1)}{\Gamma(n+\beta)}x^{\sigma\beta-1}P_n^{(\alpha+1,\beta-1)}\left(2\left(\frac{x}{b}\right)^\sigma-1\right),\\
	&&\frac{d}{dx}\left[(b^\sigma-x^{\sigma})^\alpha P_n^{(\alpha,\beta)}\left(2\left(\frac{x}{b}\right)^\sigma-1\right)\right]=\nonumber\\
	&&\hspace{1.5cm}\frac{-\sigma\Gamma(n+\alpha+1)}{\Gamma(n+\alpha)}x^{\sigma-1}(b^\sigma-x^{\sigma})^{\alpha-1}P_n^{(\alpha-1,\beta+1)}\left(2\left(\frac{x}{b}\right)^\sigma-1\right).
	\end{eqnarray*}
	\item {\bf The Rodrigues' formulas:}\\
	It is easy to verify that:
	\begin{eqnarray*}
	&&{}^1\mathcal{J}^{(\alpha,\beta,\mu,\sigma,\eta)}_n(x)=(b^\sigma-x^\sigma)^{-\alpha}x^{-\sigma(\mu+\eta+n)+n}\frac{(-1)^n}{b^{\sigma n}n!}\frac{d^n}{dx^n}\Big[(b^\sigma-x^\sigma)^{n+\alpha}x^{\sigma(n+\beta)}\Big],\\
	&&{}^2\mathcal{J}^{(\alpha,\beta,\sigma,\eta)}_n(x)=x^{-\sigma(\beta-\eta+n)+n}\frac{(-1)^n}{b^{\sigma n}n!}\frac{d^n}{dx^n}\Big[(b^\sigma-x^\sigma)^{n+\alpha}x^{\sigma(n+\beta)}\Big].
	\end{eqnarray*}
	\item{\bf Three-term recurrence formulas:}\\
	The JMFs-1 and JMFs-2 for $n\geq1$ can be generated by the following three-term recurrence formulas:
	\begin{equation*}
	A_n^{\alpha,\beta}\  {}^1\mathcal{J}^{(\alpha,\beta,\mu,\sigma,\eta)}_{n+1}(x)=\left({}^*B_n^{\alpha,\beta}x^\sigma-{}^*C_n^{\alpha,\beta}\right){}{}^1\mathcal{J}^{(\alpha,\beta,\mu,\sigma,\eta)}_{n}(x)-E_n^{\alpha,\beta}{}\ {}^1\mathcal{J}^{(\alpha,\beta,\mu,\sigma,\eta)}_{n-1}(x),
	\end{equation*} 
	and
	\begin{equation*}
			A_n^{\alpha,\beta}\  {}^2\mathcal{J}^{(\alpha,\beta,\sigma,\eta)}_{n+1}(x)=\left({}^*B_n^{\alpha,\beta}x^\sigma-{}^*C_n^{\alpha,\beta}\right){}{}^2\mathcal{J}^{(\alpha,\beta,\sigma,\eta)}_{n}(x)-E_n^{\alpha,\beta}{}\ {}^2\mathcal{J}^{(\alpha,\beta,\sigma,\eta)}_{n-1}(x),
	\end{equation*} 
where 
\begin{eqnarray*}
	&& {}^1\mathcal{J}^{(\alpha,\beta,\mu,\sigma,\eta)}_0(x)=x^{\sigma(\beta-\eta-\mu)},\\ &&{}^1\mathcal{J}^{(\alpha,\beta,\mu,\sigma,\eta)}_1(x)=x^{\sigma(\beta-\eta-\mu)}\left(\frac{1}{b^\sigma}(\alpha+\beta+2)x^\sigma+(\alpha+1)\right),\\
	&&{}^2\mathcal{J}^{(\alpha,\beta,\mu,\sigma,\eta)}_0(x)=x^{\sigma\eta}\left(b^\sigma-x^\sigma\right)^\alpha,\\
	&&{}^2\mathcal{J}^{(\alpha,\beta,\mu,\sigma,\eta)}_1(x)=x^{\sigma\eta}\left(b^\sigma-x^\sigma\right)^\alpha\left(\frac{1}{b^\sigma}(\alpha+\beta+2)x^\sigma+(\alpha+1)\right),\\
&&{}^*B_n^{\alpha,\beta}=\frac{2}{b^\sigma}B_n^{\alpha,\beta},\ {}^*C_n^{\alpha,\beta}=B_n^{\alpha,\beta}+C_n^{\alpha,\beta},
\end{eqnarray*}
and the sequences $A_n^{\alpha,\beta}$, $B_n^{\alpha,\beta}$, $C_n^{\alpha,\beta}$ and $E_n^{\alpha,\beta}$ are defined in \eqref{Eq_Rec_1}-\eqref{Eq_Rec_4}.
\item {\bf Othogonality:}\\
The orthogonality of JMFs-1 and JMFs-2 are given as follows: 
\begin{equation}\label{Prop4}
\int_{0}^{b}{}^1\mathcal{J}^{(\alpha,\beta,\mu,\sigma,\eta)}_n(x)\ {}^1\mathcal{J}^{(\alpha,\beta,\mu,\sigma,\eta)}_m(x)x^{\sigma-1}w_1^{(\alpha,\beta,\mu,\sigma,\eta)}(x)\,dx={}^*\gamma_n^{(\alpha,\beta)}\delta_{nm},
\end{equation}
and
\begin{equation}\label{Prop42}
\int_{0}^{b}{}^2\mathcal{J}^{(\alpha,\beta,\sigma,\eta)}_n(x)\ {}^2\mathcal{J}^{(\alpha,\beta,\sigma,\eta)}_m(x)x^{\sigma-1}w_2^{(\alpha,\beta,\sigma,\eta)}(x)\,dx={}^*\gamma_n^{(\alpha,\beta)}\delta_{nm},
\end{equation}
where $\displaystyle {}^*\gamma_n^{(\alpha,\beta)}=\frac{1}{\sigma}\left(\frac{b^\sigma}{2}\right)^{\alpha+\beta+1}\gamma_n^{(\alpha,\beta)}$, and $\gamma_n^{(\alpha,\beta)}$ is defined in \eqref{Orthog_Cons}.
\item {\bf Orthogonality of fractional derivatives:}\\
It is interesting to note that the following orthogonality properties for E-K fractional derivatives for JMFs-1 and JMFs-2  hold true. Let $l\in\Bbb{N}_0$, and take
\begin{equation*}
\phi_n(x)={}_{0}D_{x,\sigma,\eta}^{\mu+l}{}^1\mathcal{J}^{(\alpha,\beta,\mu,\sigma,\eta)}_n(x),\ \ \psi_n(x)={}_{x}D_{b,\sigma,\eta}^{\mu+l}{}^2\mathcal{J}^{(\alpha,\beta,\mu,\sigma,\eta)}_n(x),
\end{equation*}
then we  have:
\begin{eqnarray}\label{Prop5}
&&\ \ \ \ \  \int_{0}^{b}\phi_n(x)\phi_m(x)\ x^{\sigma-1}w_1^{(\alpha+\mu+l,\beta-\mu-l,\mu,\sigma,\eta-\mu-l)}(x)\,dx={}^1\theta_{n,l}^{(\alpha,\beta,\mu)}\delta_{nm},\label{OrthFrac_1}\\
&&\ \ \ \ \ \int_{0}^{b}\psi_n(x)\psi_m(x)\ {x}^{\sigma-1}w_2^{(\alpha-\mu-l,\beta+\mu+l,\sigma,\eta+\mu+l)}(x)\,dx={}^2\theta_{n,l}^{(\alpha,\beta,\mu)}\delta_{nm},\label{OrthFrac_2}
\end{eqnarray}
where
\begin{eqnarray}
&&\ \ \ \ \  {}^1\theta_{n,l}^{(\alpha,\beta,\mu)}=\frac{1}{\sigma}\left(\frac{b^\sigma}{2}\right)^{\alpha+\beta+1}\left(\frac{ \Gamma 
		\left( n+\beta+1 \right)} { \Gamma  \left( n+\beta-\mu-l+1\right)}\right)^2\gamma_n^{(\alpha+\mu+l,\beta-\mu-l)},\label{OrthFrac_1Cons}
 \\
&&\ \ \ \ \  {}^2\theta_{n,l}^{(\alpha,\beta,\mu)}=\frac{1}{\sigma}\left(\frac{b^\sigma}{2}\right)^{\alpha+\beta+1}\left(\frac{ \Gamma 
	\left( n+\alpha+1 \right)} { \Gamma  \left( n+\alpha-\mu-l+1\right)}\right)^2\gamma_n^{(\alpha-\mu-l,\beta+\mu+l)},\label{OrthFrac_2Cons}
\end{eqnarray}
and $\gamma_n^{(\alpha,\beta)}$ is defined in \eqref{Orthog_Cons}.
\end{itemize}
In the next subsection we will prove some important theorems to  establish our numerical methods based on the use of MJFs-1 and MJFs-2. 
\subsection{Approximation properties of JMFs-1 and JMFs-2}
The aim of this section is to study the approximation properties of the JMFs-1 and JMFs-2. To start, we introduce some definitions and notations. 

The weighted Sobolev space with respect to the weight function $\omega(x)>0$, and its norm are defined on $\Lambda=[0,b]$ by:
\begin{equation*}
{\bf L}^2_{\omega}(\Lambda)=\left\{u:\ \int_{\Lambda}u^2(x)\omega(x)\,dx<\infty\right\},\ \ \|u\|_\omega=\left(\int_{\Lambda}u^2(x)\omega(x)\,dx\right)^{\frac{1}{2}}.
\end{equation*}
Moreover, the non-uniformly  Jacobi-M\"untz spaces for $m\in\Bbb{N}_0$ are also defined:
\begin{equation}\label{NUMJS1}
{}^1{\bf B}^{m}_{\alpha,\beta,\mu,\sigma,\eta}(\Lambda):=\left\{u\in{\bf L}^2_{x^{\sigma-1}w_1^{(\alpha,\beta,\mu,\sigma,\eta)}}(\Lambda):\ {}_{0}D_{x,\sigma,\eta}^{\mu+l}u\in{\bf L}^2_{{}^*w_1^{(\alpha,\beta,\mu,\sigma,\eta)}}(\Lambda),\ 0\leq l\leq m
 \right\},
\end{equation}
and
\begin{equation}\label{NUMJS2}
{}^2{\bf B}^{m}_{\alpha,\beta,\mu,\sigma,\eta}(\Lambda):=\left\{u\in{\bf L}^2_{x^{\sigma-1}w_2^{(\alpha,\beta,\sigma,\eta)}}(\Lambda):\ {}_{x}D_{b,\sigma,\eta}^{\mu+l}u\in{\bf L}^2_{{}^*w_2^{(\alpha,\beta,\mu,\sigma,\eta)}}(\Lambda),\ 0\leq l\leq m
\right\},
\end{equation}
where
\begin{eqnarray}
&&{}^*w_1^{(\alpha,\beta,\mu,\sigma,\eta,l)}=x^{\sigma-1}w_1^{(\alpha+\mu+l,\beta-\mu-l,\mu,\sigma,\eta-\mu-l)},\label{FracWeigh_1}\\ \ &&{}^*w_2^{(\alpha,\beta,\mu,\sigma,\eta,l)}=x^{\sigma-1}w_2^{(\alpha-\mu-l,\beta+\mu+l,\sigma,\eta+\mu+l)}.\label{FracWeigh_2}
\end{eqnarray}
The finite dimensional Jacobi-M\"untz spaces are defined by:
\begin{eqnarray}\label{FDJMS}
&&{}^1\mathcal{F}^{(\alpha,\beta,\mu,\sigma,\eta)}_N(\Lambda):=\Big\{\phi:\ \phi(x)=x^{\sigma(\beta-\eta-\mu)}\psi(x^\sigma),\ \psi(x)\in\Bbb{P}_N,\ x\in\Lambda \Big\}\\
&&\hspace{2.5cm}=\text{span}\Big\{{}^1\mathcal{J}^{(\alpha,\beta,\mu,\sigma,\eta)}_n(x),\ 0\leq n\leq N,\ x\in\Lambda \Big\}, \label{FDJMS1}\nonumber \\
&&{}^2\mathcal{F}^{(\alpha,\beta,\sigma,\eta)}_N(\Lambda):=\Big\{\phi:\ \phi(x)=x^{\sigma\eta}(b^\sigma-x^\sigma)^\alpha\psi(x^\sigma),\ \psi(x)\in\Bbb{P}_N,\ x\in\Lambda \Big\}\\
&&\hspace{2.5cm}=\text{span}\Big\{{}^2\mathcal{J}^{(\alpha,\beta,\sigma,\eta)}_n(x),\ 0\leq n\leq N,\ x\in\Lambda,\  \Big\},\label{FDJMS2}\nonumber
\end{eqnarray}
where $\Bbb{P}_N$ stands for the set of polynomials of degree $\leq N$. 

Now, let $u\in{\bf L}^2_{w^{(\alpha,\beta,\sigma)}}(\Lambda)$, where $w^{(\alpha,\beta,\sigma)}(x)=x^{\sigma-1}w^{(\alpha,\beta)}\left(2\left(\frac{x}{b}\right)^\sigma-1\right)$ and note that $w^{(\alpha,\beta)}=(1-x)^\alpha(1+x)^\beta$. Then we can quickly expand $u(x)$ as follows:
\begin{equation}
u(x)=\sum_{k=0}^{\infty}\bar{u}_k^{(\alpha,\beta,\sigma)}P_k^{(\alpha,\beta)}\left(2\left(\frac{x}{b}\right)^\sigma-1\right),
\end{equation}
where
\begin{equation}
\bar{u}_k^{(\alpha,\beta,\sigma)}=\sigma\left(\frac{2}{b^\sigma}\right)^{\alpha+\beta+1}\frac{1}{\gamma_k^{(\alpha,\beta)}}\int_{0}^{b}u(x)P_k^{(\alpha,\beta)}\left(2\left(\frac{x}{b}\right)^\sigma-1\right)w^{(\alpha,\beta,\sigma)}(x)\,dx.
\end{equation}
and $\gamma_k^{(\alpha,\beta)}$ is defined in \eqref{Orthog_Cons}. We also have the well known Parseval identity \cite{MR2867779}:
\begin{equation}
\|u\|_{w^{(\alpha,\beta,\sigma)}}^2=\frac{1}{\sigma}\left(\frac{b^\sigma}{2}\right)^{\alpha+\beta+1}\sum_{k=0}^{\infty}\gamma_k^{(\alpha,\beta)}\left|\bar{u}_k^{(\alpha,\beta,\sigma)}\right|^2.
\end{equation}
Now, the next theorem states the completeness of the JMFs-1 and JMFs-2 in some suitable spaces.
\begin{theorem}\label{Complete}
Let $\alpha,\beta>-1$. The  sets of JMFs $\left\{{}^1\mathcal{J}^{(\alpha,\beta,\mu,\sigma,\eta)}_n(x)\right\}_{n=0}^{\infty}$ and $\left\{{}^2\mathcal{J}^{(\alpha,\beta,\sigma,\eta)}_n(x)\right\}_{n=0}^{\infty}$ construct two complete sets in spaces ${\bf L}^2_{x^{\sigma-1}w_1^{(\alpha,\beta,\mu,\sigma,\eta)}}(\Lambda)$ and ${\bf L}^2_{x^{\sigma-1}w_2^{(\alpha,\beta,\sigma,\eta)}}(\Lambda)$, respectively.
\end{theorem}
\begin{proof}
	We start to prove the completeness of $\left\{{}^1\mathcal{J}^{(\alpha,\beta,\mu,\sigma,\eta)}_n(x)\right\}_{n=0}^{\infty}$ in corresponding ${\bf L}^2_{x^{\sigma-1}w_1^{(\alpha,\beta,\mu,\sigma,\eta)}}(\Lambda)$. The proof of the second set is fairly similar to the proof of the first one. Let $u\in{\bf L}^2_{x^{\sigma-1}w_1^{(\alpha,\beta,\mu,\sigma,\eta)}}(\Lambda)$, then we have $x^{\sigma(\eta+\mu-\beta)}u\in{\bf L}^2_{w^{(\alpha,\beta,\sigma)}}(\Lambda)$, where $w^{(\alpha,\beta,\sigma)}(x)=x^{\sigma-1}w^{(\alpha,\beta)}\left(2\left(\frac{x}{b}\right)^\sigma-1\right)$. Now, thanks to the fact that $\displaystyle\left\{P_n^{(\alpha,\beta)}\left(2\left(\frac{x}{b}\right)^\sigma-1\right)\right\}_{n=0}^\infty$ are mutually orthogonal and also complete in the space ${\bf L}^2_{w^{(\alpha,\beta,\sigma)}}(\Lambda)$, then we can expand  $x^{\sigma(\eta+\mu-\beta)}u$ as follows:
	\begin{equation}\label{Expand_1}
x^{\sigma(\eta+\mu-\beta)}u(x)=\sum_{k=0}^{\infty}\bar{v}_k^{(\alpha,\beta,\sigma)}P_k^{(\alpha,\beta)}\left(2\left(\frac{x}{b}\right)^\sigma-1\right),
	\end{equation} 
	where
	\begin{eqnarray*}
\bar{v}_k^{(\alpha,\beta,\sigma)}&=&\sigma\left(\frac{2}{b^\sigma}\right)^{\alpha+\beta+1}\frac{1}{\gamma_k^{(\alpha,\beta)}}\int_{0}^{b}x^{\sigma(\eta+\mu-\beta)}u(x)P_k^{(\alpha,\beta)}\left(2\left(\frac{x}{b}\right)^\sigma-1\right)w^{(\alpha,\beta,\sigma)}(x)\,dx\\
&=&\sigma\left(\frac{2}{b^\sigma}\right)^{\alpha+\beta+1}\frac{1}{\gamma_k^{(\alpha,\beta)}}\int_{0}^{b}u(x){}^1\mathcal{J}^{(\alpha,\beta,\mu,\sigma,\eta)}_k(x)x^{\sigma-1}w_1^{(\alpha,\beta,\mu,\sigma,\eta)}\,dx.
	\end{eqnarray*}
Now, multiplying \eqref{Expand_1} by $x^{-\sigma(\eta+\mu-\beta)}$ and the use of the last term of the above relation, we find that the set $\left\{{}^1\mathcal{J}^{(\alpha,\beta,\mu,\sigma,\eta)}_n(x)\right\}_{n=0}^{\infty}$ is completed in corresponding ${\bf L}^2_{x^{\sigma-1}w_1^{(\alpha,\beta,\mu,\sigma,\eta)}}(\Lambda)$.
\end{proof}

In this position, we are ready to introduce two important concepts in the spectral methods which are renowned as the ${\bf L}^2_{x^{\sigma-1}w_1^{(\alpha,\beta,\mu,\sigma,\eta)}}(\Lambda)$ and ${\bf L}^2_{x^{\sigma-1}w_2^{(\alpha,\beta,\sigma,\eta)}}(\Lambda)$-orthogonal projection on ${}^1\mathcal{F}^{(\alpha,\beta,\mu,\sigma,\eta)}_N(\Lambda)$ and ${}^2\mathcal{F}^{(\alpha,\beta,\sigma,\eta)}_N(\Lambda)$, respectively.
\begin{definition}\label{OrthProj}
	The ${\bf L}^2_{x^{\sigma-1}w_1^{(\alpha,\beta,\mu,\sigma,\eta)}}(\Lambda)$ and ${\bf L}^2_{x^{\sigma-1}w_2^{(\alpha,\beta,\sigma,\eta)}}(\Lambda)$-orthogonal projection on ${}^1\mathcal{F}^{(\alpha,\beta,\mu,\sigma,\eta)}_N(\Lambda)$ and ${}^2\mathcal{F}^{(\alpha,\beta,\sigma,\eta)}_N(\Lambda)$ are defined by:
	\begin{equation}
	\left({}^1\pi^{{(\alpha,\beta,\mu,\sigma,\eta)}}_Nu-u,v_N\right)_{x^{\sigma-1}w_1^{(\alpha,\beta,\mu,\sigma,\eta)}}=0,\ \ \forall v_N\in{}^1\mathcal{F}^{(\alpha,\beta,\mu,\sigma,\eta)}_N(\Lambda),
	\end{equation}
	and 
	\begin{equation}
	\left({}^2\pi^{{(\alpha,\beta,\sigma,\eta)}}_Nu-u,v_N\right)_{x^{\sigma-1}w_2^{(\alpha,\beta,\sigma,\eta)}}=0,\ \ \forall v_N\in{}^2\mathcal{F}^{(\alpha,\beta,\sigma,\eta)}_N(\Lambda),
	\end{equation}
	respectively. By definition, we immediately arrive at:
	\begin{eqnarray}
&&{}^1\pi^{{(\alpha,\beta,\mu,\sigma,\eta)}}_Nu(x)=\sum_{k=0}^{N}\hat{u}_k^{(\alpha,\beta,\mu,\sigma,\eta)}{}\ {}^1\mathcal{J}^{(\alpha,\beta,\mu,\sigma,\eta)}_k(x),\\
&& {}^2\pi^{{(\alpha,\beta,\sigma,\eta)}}_Nu(x)=\sum_{k=0}^{N}\hat{u}_k^{(\alpha,\beta,\sigma,\eta)}{}\ {}^2\mathcal{J}^{(\alpha,\beta,\sigma,\eta)}_k(x)
	\end{eqnarray}
	\end{definition}
	
An important question from the numerical analysis viewpoint  which
remains to be answered here is that: How fast the coefficients $\hat{u}_k^{(\alpha,\beta,\mu,\sigma,\eta)}$ and $\hat{u}_k^{(\alpha,\beta,\sigma,\eta)}$ decay? 

In the next theorem, we will answer the mentioned question. In the rest of this paper, we use $c$ to be a generic constant.
\begin{theorem}\label{ErrorBounds}
Let $\alpha,\beta>-1$ and  $u\in {}^i{\bf B}^{m}_{\alpha,\beta,\mu,\sigma,\eta}(\Lambda)$ with $m\in\Bbb{N}_0$, then

\begin{itemize}
	\item For $0\leq l< m\leq N$, we have:
\begin{eqnarray}
&&\left\|{}_{0}D_{x,\sigma,\eta}^{\mu+l}\left({}^1\pi^{{(\alpha,\beta,\mu,\sigma,\eta)}}_Nu-u\right)\right\|_{{}^*w_1^{(\alpha,\beta,\mu,\sigma,\eta,m)}} \hspace{4.5cm}\nonumber\\
&&\hspace{2.8cm}\leq N^{\frac{l-m}{2}}\ \sqrt{\frac{\Gamma(N+\beta-\mu-m+2)}{\Gamma(N+\beta-\mu-l+2)}}\left\|{}_{0}D_{x,\sigma,\eta}^{\mu+m}u\right\|_{{}^*w_1^{(\alpha,\beta,\mu,\sigma,\eta,m)}},
\end{eqnarray}
\begin{eqnarray}
&&\left\|{}_{x}D_{b,\sigma,\eta}^{\mu+l}\left({}^2\pi^{{(\alpha,\beta,\sigma,\eta)}}_Nu-u\right)\right\|_{{}^*w_2^{(\alpha,\beta,\mu,\sigma,\eta,m)}}\hspace{4.5cm}\nonumber\\
&&\hspace{2.8cm}\leq N^{\frac{l-m}{2}}\ \sqrt{\frac{\Gamma(N+\alpha-\mu-m+2)}{\Gamma(N+\alpha-\mu-l+2)}}\left\|{}_{x}D_{b,\sigma,\eta}^{\mu+m}u\right\|_{{}^*w_2^{(\alpha,\beta,\mu,\sigma,\eta,m)}}.
\end{eqnarray}
\item For fixed $m$, we find that:
\begin{equation}
\left\|{}_{0}D_{x,\sigma,\eta}^{\mu+l}\left({}^1\pi^{{(\alpha,\beta,\mu,\sigma,\eta)}}_Nu-u\right)\right\|_{x^{\sigma-1}w_1^{(\alpha,\beta,\mu,\sigma,\eta)}} \leq c N^{{l-m}}\left\|{}_{0}D_{x,\sigma,\eta}^{\mu+m}u\right\|_{{}^*w_1^{(\alpha,\beta,\mu,\sigma,\eta,m)}},
\end{equation}
\begin{equation}
\left\|{}_{x}D_{b,\sigma,\eta}^{\mu+l}\left({}^2\pi^{{(\alpha,\beta,\sigma,\eta)}}_Nu-u\right)\right\|_{x^{\sigma-1}w_2^{(\alpha,\beta,\sigma,\eta)}}\leq c N^{l-m}\left\|{}_{x}D_{b,\sigma,\eta}^{\mu+m}u\right\|_{{}^*w_2^{(\alpha,\beta,\mu,\sigma,\eta,m)}}.
\end{equation}
 \item For $0\leq m \leq N$ we also have:
\begin{eqnarray}
&&\|{}^1\pi^{{(\alpha,\beta,\mu,\sigma,\eta)}}_Nu-u\|_{x^{\sigma-1}w_1^{(\alpha,\beta,\mu,\sigma,\eta)}}\nonumber\\
&&\hspace{2.8cm}\leq c N^{\frac{\alpha-\beta}{2}}\ \sqrt{\frac{\Gamma(N+\beta-\mu-m+2)}{\Gamma(N+\alpha+\mu+m+2)}}\left\|{}_{0}D_{x,\sigma,\eta}^{\mu+m}u\right\|_{{}^*w_1^{(\alpha,\beta,\mu,\sigma,\eta,m)}},
\end{eqnarray}
\begin{eqnarray}
&&\|{}^2\pi^{{(\alpha,\beta,\sigma,\eta)}}_Nu-u\|_{x^{\sigma-1}w_2^{(\alpha,\beta,\sigma,\eta)}}\nonumber\\
&&\hspace{2.8cm}\leq c N^{\frac{\beta-\alpha}{2}}\ \sqrt{\frac{\Gamma(N+\alpha-\mu-m+2)}{\Gamma(N+\beta+\mu+m+2)}}\left\|{}_{x}D_{b,\sigma,\eta}^{\mu+m}u\right\|_{{}^*w_2^{(\alpha,\beta,\mu,\sigma,\eta,m)}},
\end{eqnarray}
\item For fixed $m$ we also have:
\begin{eqnarray*}
&&\|{}^1\pi^{{(\alpha,\beta,\mu,\sigma,\eta)}}_Nu-u\|_{x^{\sigma-1}w_1^{(\alpha,\beta,\mu,\sigma,\eta)}}\leq c N^{-(m+\mu)}\left\|{}_{0}D_{x,\sigma,\eta}^{\mu+m}u\right\|_{{}^*w_1^{(\alpha,\beta,\mu,\sigma,\eta,m)}},\\
&&\|{}^2\pi^{{(\alpha,\beta,\sigma,\eta)}}_Nu-u\|_{x^{\sigma-1}w_2^{(\alpha,\beta,\sigma,\eta)}}\leq c N^{-(\mu+m)}\ \left\|{}_{x}D_{b,\sigma,\eta}^{\mu+m}u\right\|_{{}^*w_2^{(\alpha,\beta,\mu,\sigma,\eta,m)}},
\end{eqnarray*}
\end{itemize}
\end{theorem}
\begin{proof}
By noting \eqref{OrthFrac_1},  \eqref{FracWeigh_1} and \eqref{OrthFrac_1Cons}, we immediately arrive at:
\begin{eqnarray*}
&&\left\|{}_{0}D_{x,\sigma,\eta}^{\mu+l}\left({}^1\pi^{{(\alpha,\beta,\mu,\sigma,\eta)}}_Nu-u\right)\right\|_{{}^*w_1^{(\alpha,\beta,\mu,\sigma,\eta,l)}}^2=\sum_{k=N+1}^{\infty}\left|\hat{u}_k^{(\alpha,\beta,\mu,\sigma,\eta)}\right|^2{}^1\theta_{k,l}^{(\alpha,\beta,\mu)}\\
&=&\sum_{k=N+1}^{\infty}\left|\hat{u}_k^{(\alpha,\beta,\mu,\sigma,\eta)}\right|^2\left(\frac{{}^1\theta_{k,l}^{(\alpha,\beta,\mu)}}{{}^1\theta_{k,m}^{(\alpha,\beta,\mu)}}\right){}^1\theta_{k,m}^{(\alpha,\beta,\mu)}\leq \frac{{}^1\theta_{N+1,l}^{(\alpha,\beta,\mu)}}{{}^1\theta_{N+1,m}^{(\alpha,\beta,\mu)}}\left\|{}_{0}D_{x,\sigma,\eta}^{\mu+m}u\right\|_{{}^*w_1^{(\alpha,\beta,\mu,\sigma,\eta,m)}}^2.
\end{eqnarray*}
Now, it remains to estimate the coefficient $\displaystyle \frac{{}^1\theta_{N+1,l}^{(\alpha,\beta,\mu)}}{{}^1\theta_{N+1,m}^{(\alpha,\beta,\mu)}}$. The use of \eqref{PochSymb}, \eqref{Orthog_Cons} and \eqref{OrthFrac_1Cons}, we get:
\begin{equation}\label{EQQ}
\frac{{}^1\theta_{N+1,l}^{(\alpha,\beta,\mu)}}{{}^1\theta_{N+1,m}^{(\alpha,\beta,\mu)}}=\frac{\Gamma(N+\beta-\mu-m+2)\Gamma(N+\alpha+\mu+l+2)}{\Gamma(N+\beta-\mu-l+2)\Gamma(N+\alpha+\mu+m+2)},
\end{equation}
thanks to the fact that:
\begin{eqnarray*}
\frac{\Gamma(N+\alpha+\mu+l+2)}{\Gamma(N+\alpha+\mu+m+2)}&=&\frac{1}{(N+\alpha+\mu+m+1)(N+\alpha+\mu+m)\cdots(N+\alpha+\mu+l+2)}\nonumber\\
&\leq& N^{l-m},
\end{eqnarray*}
plugging the above upper bound in Equation \eqref{EQQ}, we arrive at:
\begin{equation}\label{EQQ1}
\frac{{}^1\theta_{N+1,l}^{(\alpha,\beta,\mu)}}{{}^1\theta_{N+1,m}^{(\alpha,\beta,\mu)}}\leq N^{l-m}\ \frac{\Gamma(N+\beta-\mu-m+2)}{\Gamma(N+\beta-\mu-l+2)},
\end{equation}
which completes the proof. If $m$ is fixed, then thanks to the asymptotic formula \eqref{Asymp1}, we immediately find that: 
\begin{equation}\label{EQQ1Asymp}
\frac{{}^1\theta_{N+1,l}^{(\alpha,\beta,\mu)}}{{}^1\theta_{N+1,m}^{(\alpha,\beta,\mu)}}\leq c N^{2(m-l)}.
\end{equation} 
This completes the proof.

Now, we start to estimate:
	\begin{eqnarray*}
&&\|{}^1\pi^{{(\alpha,\beta,\mu,\sigma,\eta)}}_Nu-u\|^2_{x^{\sigma-1}w_1^{(\alpha,\beta,\mu,\sigma,\eta)}}=\sum_{k=N+1}^{\infty}|\hat{u}_k^{(\alpha,\beta,\mu,\sigma,\eta)}{}|^2\ {}^*\gamma_k^{(\alpha,\beta)}\\
&=&\sum_{k=N+1}^{\infty}\left|\hat{u}_k^{(\alpha,\beta,\mu,\sigma,\eta)}\right|^2\left(\frac{{}^*\gamma_k^{(\alpha,\beta)}}{{}^1\theta_{k,m}^{(\alpha,\beta,\mu)}}\right){}^1\theta_{k,m}^{(\alpha,\beta,\mu)}\leq \frac{{}^*\gamma_{N+1}^{(\alpha,\beta)}}{{}^1\theta_{N+1,m}^{(\alpha,\beta,\mu)}}\left\|{}_{0}D_{x,\sigma,\eta}^{\mu+m}u\right\|_{{}^*w_1^{(\alpha,\beta,\mu,\sigma,\eta,m)}}^2,
	\end{eqnarray*}
	where by using the asymptotic formula \eqref{Asymp1}, we easily find that:
\begin{eqnarray*}\label{EQQL2}
&&\frac{{}^*\gamma_{N+1}^{(\alpha,\beta)}}{{}^1\theta_{N+1,m}^{(\alpha,\beta,\mu)}}=\frac{\Gamma(N+\alpha+2)\Gamma(N+\beta-\mu-m+2)}{\Gamma(N+\beta+2)\Gamma(N+\alpha+\mu+m+2)}\nonumber\\
&&\hspace{1.5cm}\leq c N^{\alpha-\beta}\frac{\Gamma(N+\beta-\mu-m+2)}{\Gamma(N+\alpha+\mu+m+2)}\leq cN^{-2(m+\mu)},\ \text{(if $m$ is fixed)},
\end{eqnarray*}
which concludes the proof.	
\end{proof}
\begin{remark}
An important issue which we emphasize here is that the previous theorem, in fact, states that the orthogonal projection ${}^i\pi^{{(\alpha,\beta,\mu,\sigma,\eta)}}_Nu$ is the best approximation of $u$ in both spaces  ${\bf L}^2_{x^{\sigma-1}w_i^{(\alpha,\beta,\mu,\sigma,\eta)}}(\Lambda)$	 and ${}^i{\bf B}^{m}_{\alpha,\beta,\mu,\sigma,\eta}(\Lambda)$.  
\end{remark}
\subsection{Gauss-Jacobi-M\"untz quadrature rules}
In this subsection, two new quadrature rules based on JMFs-1 and JMFs-2 are introduced. To do so, we denote:
\begin{eqnarray}
&&\Bbb{P}_{N}^{(\sigma)}:=\text{span}\left\{x^{k\sigma}:\ k=0,1,\ldots,N\right\},\label{MuntzSp1}\\
&&\Bbb{P}_{N}^{(\beta,\mu,\sigma,\eta)}:=\text{span}\left\{x^{2\sigma(\beta-\mu-\eta)+k\sigma}:\ k=0,1,\ldots,N\right\},\label{MuntzSp2}\\
&&\Bbb{P}_{N}^{(\alpha,\sigma,\eta)}:=\text{span}\left\{(b^\sigma-x^\sigma)^{2\alpha}x^{2\sigma\eta+k\sigma}:\ k=0,1,\ldots,N\right\}.\label{MuntzSp3}
\end{eqnarray}
In the next theorem two new quadrature rules based on the JMFs-1 and JMFs-1 are presented.
\begin{theorem}\label{MintzJacQuad}
	Let $\sigma>0$ and $\alpha,\beta>-1$. Let  $x_j^{(\alpha,\beta)}$ and $w_j^{(\alpha,\beta)}$ for $j=0,1,2\ldots,n$ be the Gauss-Jacobi nodes and weights with parameter $(\alpha,\beta)$ on $[-1,1]$, respectively. Then we have the following quadrature rule:
\begin{equation}\label{MuntsQuad}
	\int_{0}^{b}f(x)x^{\sigma(\beta+1)-1}(b^{\sigma}-x^{\sigma})^{\alpha}\,dx=\sum_{j=0}^{n}w_j^{(\alpha,\beta,\sigma)}f\left(x_j^{(\alpha,\beta,\sigma)}\right)+E_n[f],
\end{equation}
where $E_n[f]$ stands for the quadrature error. Then the above quadrature formula is exact (i.e., $E_n[f]=0$) for any $f(x)\in\Bbb{P}_{2n+1}^{(\sigma)}$, where
\begin{equation}\label{MuntsQuadNodWei}
w_j^{(\alpha,\beta,\sigma)}=\frac{1}{\sigma}\left(\frac{b^\sigma}{2}\right)^{\alpha+\beta+1}w_j^{(\alpha,\beta)},\ \ \ \ \ \ \  x_j^{(\alpha,\beta,\sigma)}=b\left(\frac{1+x_j^{(\alpha,\beta)}}{2}\right)^{\frac{1}{\sigma}}. 
\end{equation}
Also, the Gauss-Jacobi-M\"untz quadrature rules of the first and second types (which are denoted respectively by GJMQR-1 and GJMQR-2) are as follows:
\begin{equation}\label{JacMunts1Quad}
\int_{0}^{b}f(x)x^{\sigma(2(\eta+\mu)-\beta+1)-1}\left(b^\sigma-x^\sigma\right)^{\alpha}\,dx=\sum_{j=0}^{n}w_j^{(\alpha,\beta,\mu,\sigma,\eta)}f\left(x_j^{(\alpha,\beta,\sigma)}\right)+{}^1E_n[f],
\end{equation}
and
\begin{equation}\label{JacMunts2Quad}
\int_{0}^{b}f(x)x^{\sigma(\beta-2\eta+1)-1}\left(b^\sigma-x^\sigma\right)^{-\alpha}\,dx=\sum_{j=0}^{n}w_j^{(\alpha,\beta,\sigma,\eta)}f\left(x_j^{(\alpha,\beta,\sigma)}\right)+{}^2E_n[f].
\end{equation}
 The above quadrature formulas \eqref{JacMunts1Quad} and \eqref{JacMunts2Quad} are exact (i.e., ${}^iE_n[f]=0$) for any $f(x)\in\Bbb{P}_{2n+1}^{(\beta,\mu,\sigma,\eta)}$ and $f(x)\in\Bbb{P}_{2n+1}^{(\alpha,\sigma,\eta)}$, where
 \begin{eqnarray}\label{JacMunts1-2QuadNodWei}
 &&w_j^{(\alpha,\beta,\mu,\sigma,\eta)}={w_j^{(\alpha,\beta,\sigma)}}\ {\left(x_j^{(\alpha,\beta,\sigma)}\right)^{2\sigma(\eta+\mu-\beta)}},\ \ \ \ \ \ \\
 && w_j^{(\alpha,\beta,\sigma,\eta)}= w_j^{(\alpha,\beta,\sigma)}}\  {\left(b^\sigma-\left(x_j^{(\alpha,\beta,\sigma)}\right)^{\sigma}\right)^{-2\alpha}}{\left(x_j^{(\alpha,\beta,\sigma)}\right)^{-2\sigma\eta}.
 \end{eqnarray}
\end{theorem}
\begin{proof}
	The proof is straightforward. 
\end{proof}
It should be noted that the nodes and weights of the newly generated quadrature rules are dependent on two parameters $b$ and $\sigma$.  Moreover, the nodes and weights $x_j^{(\alpha,\beta,\sigma)}$, $w_j^{(\alpha,\beta,\sigma)}$ for $\sigma=1$ reduce to the classical Jacobi-Gauss nodes and weights on $[0,b]$ (see blue filled circle in \cref{Fig-1}). 

A natural question comes to our mind is that: What is the effect of the parameter $\sigma$ on the distribution of the Jacobi-Gauss nodes and weights $x_j^{(\alpha,\beta)}$ and $w_j^{(\alpha,\beta)}$ on $[-1,1]$?

To answer this question,  the quadratures' nodes and 
weights $x_j^{(\alpha,\beta,\sigma)}$, $w_j^{(\alpha,\beta,\sigma)}$, $w_j^{(\alpha,\beta,\mu,\sigma,\eta)}$ and $w_j^{(\alpha,\beta,\sigma,\eta)}$ with values $\alpha=0.5,\ \beta=1.5,\ \eta=2,\ \mu=0.5$, $n=50$ for some values of $\sigma\in(0,2]$ on the domain $[0,10]$ are plotted in \cref{Fig-1}. 
\begin{figure}[htbp]
	\vspace{-2cm}
	\centering
	\includegraphics[width=6.5cm,height=8.5cm,keepaspectratio=true]{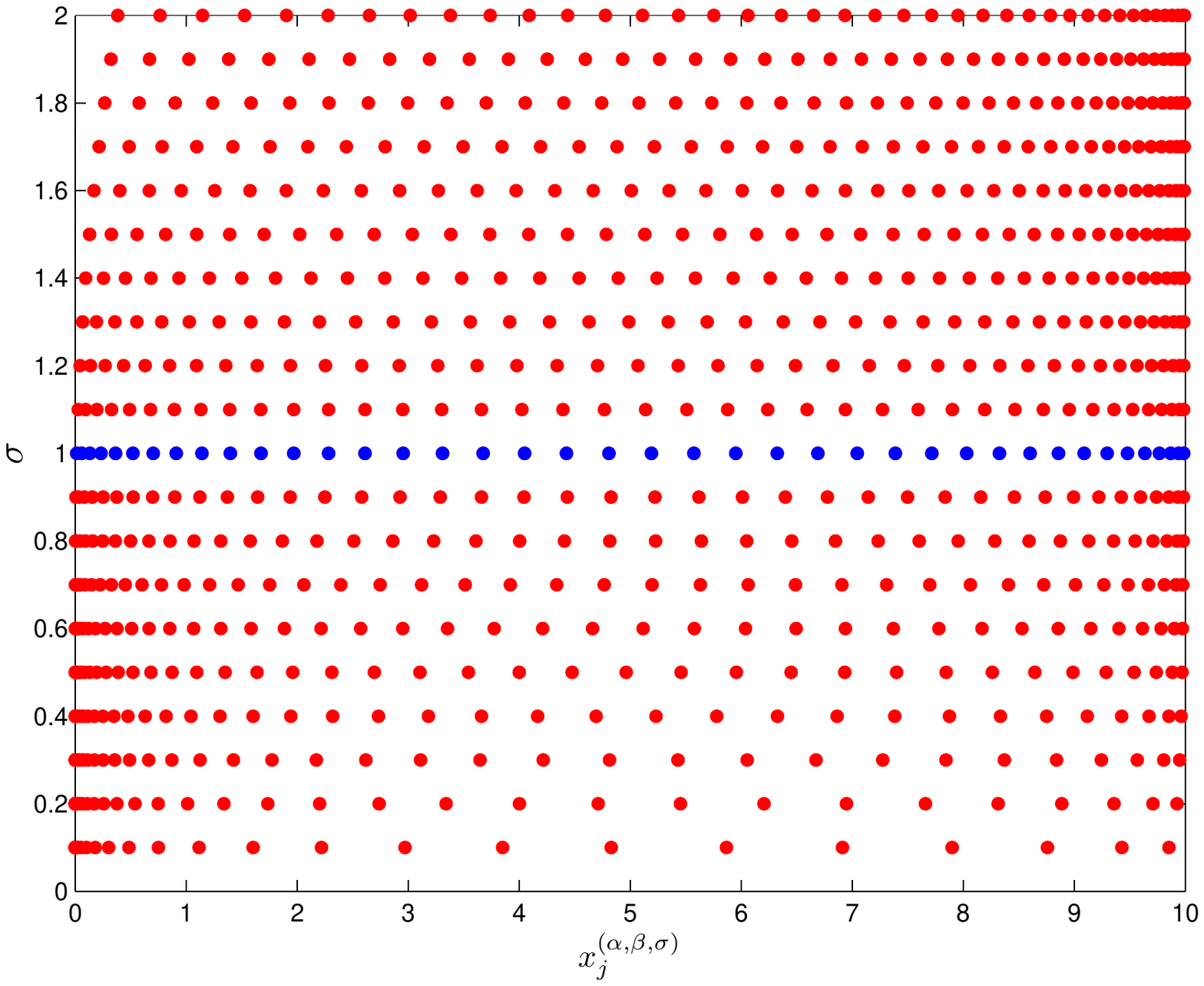}\includegraphics[width=6.5cm,height=8.5cm,keepaspectratio=true]{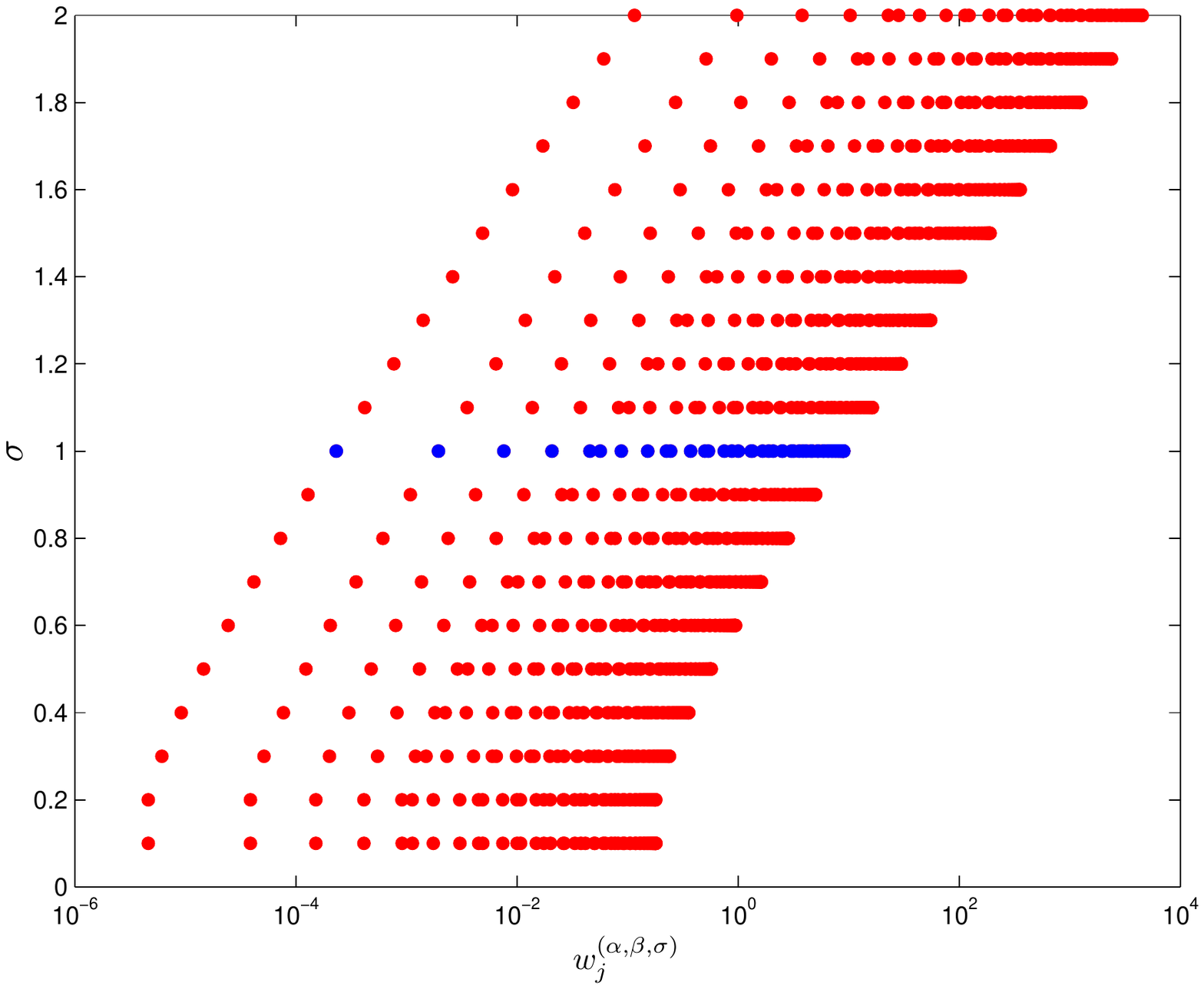}\vspace{-3.8cm}
	\includegraphics[width=6.5cm,height=8.5cm,keepaspectratio=true]{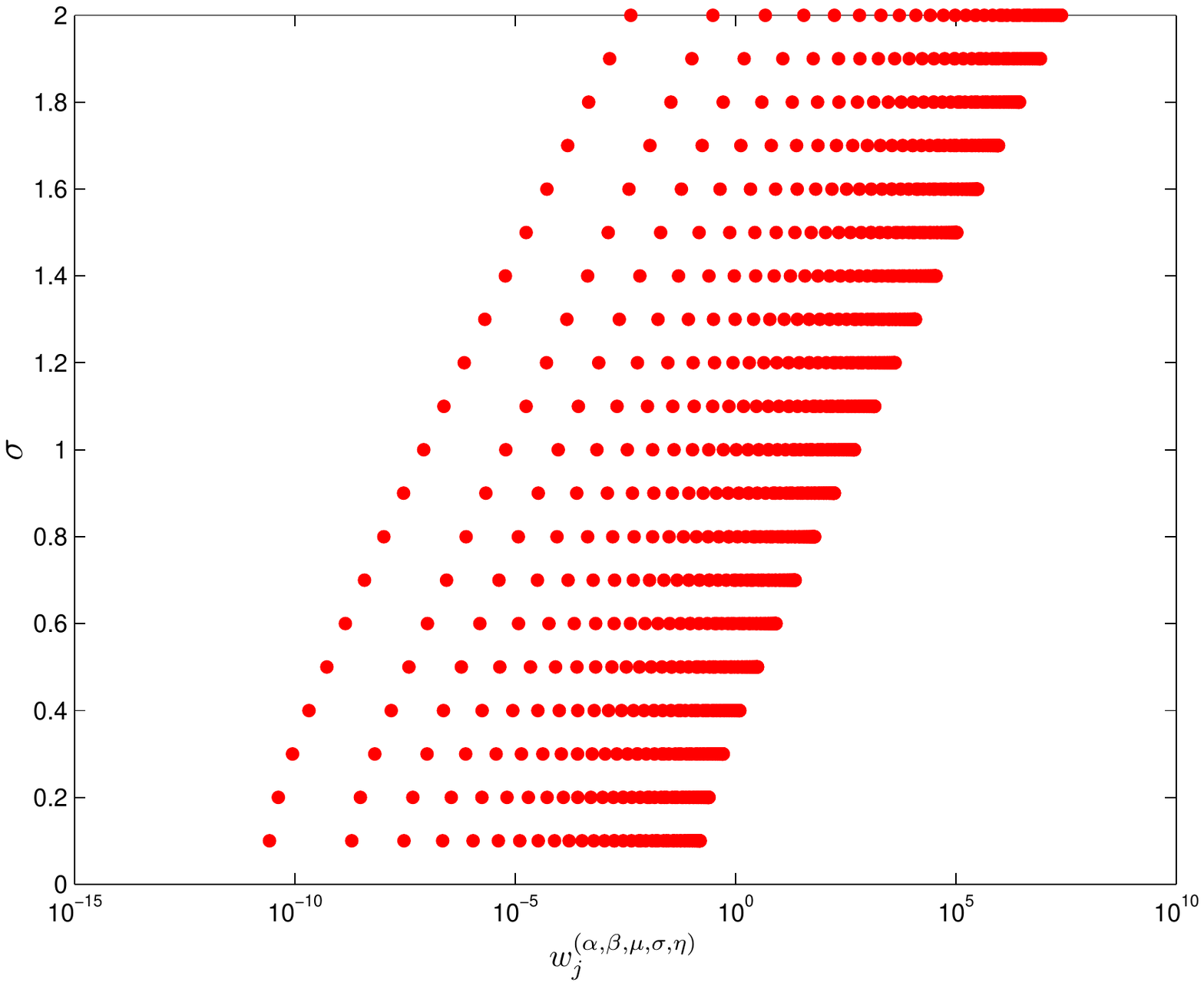}\includegraphics[width=6.5cm,height=8.5cm,keepaspectratio=true]{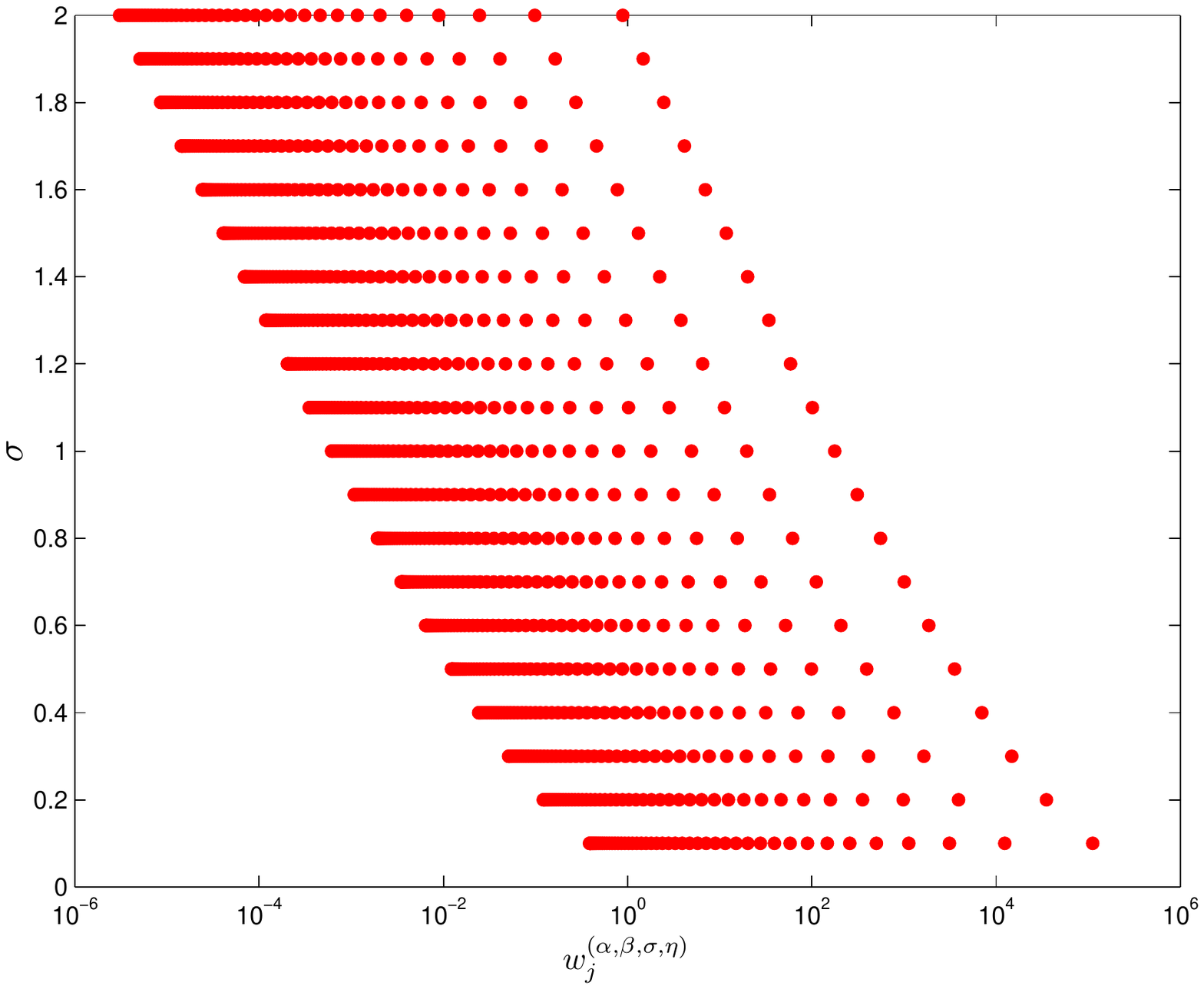}
	\vspace{-2.5cm}
	\caption{The distribution of the newly generated quadratures' nodes and weights (in linear-log scale) for various values of $\sigma\in(0,2]$.}
	\label{Fig-1}
\end{figure}

We can observe from \cref{Fig-1} that when the parameter $\sigma$ tends to $0$, then the nodes are clustered to $x=0$ and when $\sigma$ tends from $1$ to $2$, then the nodes are clustered to $x=b$. Therefore, the $\sigma$ parameter can give the opportunity for the users to cluster the collocation points according to their needs.

\section{Some applications of JMFs-1 and JMFs-2}\label{sec:experiments}
This section is devoted to provide some applications of the introduced basis functions. So, this section is divided into two subsections. The first part is:   Application to fractional differential equations and the second part is: Application to fractional partial differential equations.
\subsection{Application to fractional differential equations}
As the first application of JMFs-1 and JMFs-2, we will use JMFs-1  to solve a simple fractional differential equation. 
\begin{example}\label{Ex-1}
	For the first example, consider the following steady-state fractional differential equation:
	\begin{equation}\label{EqEx_1}
	K_2\ {}_{0}D_{x,\sigma,\eta}^{\mu}y(x)	+K_1\  y(x)=f(x),\ y(0)=y'(0)=0,\  1<\mu<2.
	\end{equation}
	To solve this problem numerically, we approximate the exact solution as follows:
	\begin{equation}\label{Spect_1}
y(x)\approx y_n(x)=\sum_{k=0}^{n} a_k\ {}^1\mathcal{J}^{(\alpha,\beta,\mu,\sigma,\eta)}_k(x),
	\end{equation}
	where the parameters $\beta,\ \mu,\ \eta$ are chosen such that $y_n(0)=y_n'(0)=0$. By substituting $y_n(x)$ into \cref{EqEx_1} and collocating both sides at $\{x_j\}_{j=0}^n=\left\{x_j^{(\alpha,\beta,\sigma)}\right\}_{j=0}^n$ which is defined in \eqref{MuntsQuadNodWei}, we get:
	\begin{equation}
	\left(K_2\ {\bf D^\mu}+K_1\ {\bf M}\right){\bf a}={\bf F},
	\end{equation}
	where for $j,k=0,1,2,\ldots,n,$, we have:
		\begin{eqnarray}
		&&{\bf M}=(m_{j,k}),\  m_{j,k}={}^1\mathcal{J}^{(\alpha,\beta,\mu,\sigma,\eta)}_k(x_j),\ \label{MatM}\\
		&&
		{\bf D^\mu}=(d_{j,k})=\frac{\Gamma(k+\beta+1)}{\Gamma(k+\beta-\mu+1)}{}^1\mathcal{J}^{(\alpha+\mu,\beta-\mu,\mu,\sigma,\eta-\mu)}_k(x_j),\label{MatDiffM}\\
		&& {\bf a}=[a_0\ a_1 \ldots\ a_n]^T,\ {\bf F}=[f(x_0)\ f(x_1) \ldots\ f(x_n)]^T.
		\end{eqnarray}
		Finally, the approximate solution ${\bf Y}$ is obtained as:
	\begin{equation}
{\bf Y}={\bf M}\,{\bf a},\ \text{where} \  {\bf Y}=[y(x_0)\ y(x_1) \ldots\ y(x_n)]^T,\ \ {\bf a}=\left(K_2\ {\bf D^\mu}+K_1\ {\bf M}\right)^{-1}{\bf F}.
	\end{equation} 
Here, we solve the problem \cref{EqEx_1} numerically when that the exact solution  is $y(x)={x}^{\sigma\,\nu}+7\,{x}^{2\,\sigma\,\nu}$.

For a more accurate comparison, we also solve this problem by the use of the corresponding M\"untz basis functions $\left\{x^{\sigma(\beta-\eta-\mu+k)}\right\}_{k=0}^n$. Maximum errors together with the condition numbers of the coefficient matrix  for various values of $n=1:1:100$ with $\ \alpha=0.5,\ \beta=1.5,\ \eta=-3,\ \sigma=0.5,\ \mu=1.5,\ b=1$ and $\nu=3$ are shown in  the first row of \cref{Fig-3}. In the second row of \cref{Fig-3}, the Fourier coefficients $a_k,\  k=1:1:n$ with $n=100$ for both cases JMFs-1 and corresponding  M\"untz basis functions for $ \alpha=0.5,\ \beta=1.5,\ \eta=-3,\ \sigma=0.5,\ \mu=1.5,\  \nu=3$ when $x\in[0,1]$ are plotted. The following facts can be clearly observed  from \cref{Fig-3}:
\begin{itemize}
\item The use of JMFs-1 not only (may) leads to have a stable numerical scheme but also it reduces the condition numbers of the coefficient matrix substantially. 
\item The use of the M\"untz basis functions $\left\{x^{\sigma(\beta-\eta-\mu+k)}\right\}_{k=0}^n$ may results that their Fourier coefficients $a_k$ are evaluated in an unstable manner. 
\end{itemize} 
\begin{figure}[htbp]
	\vspace{-2.5cm}
	\centering
	\includegraphics[width=6.5cm,height=8.5cm,keepaspectratio=true]{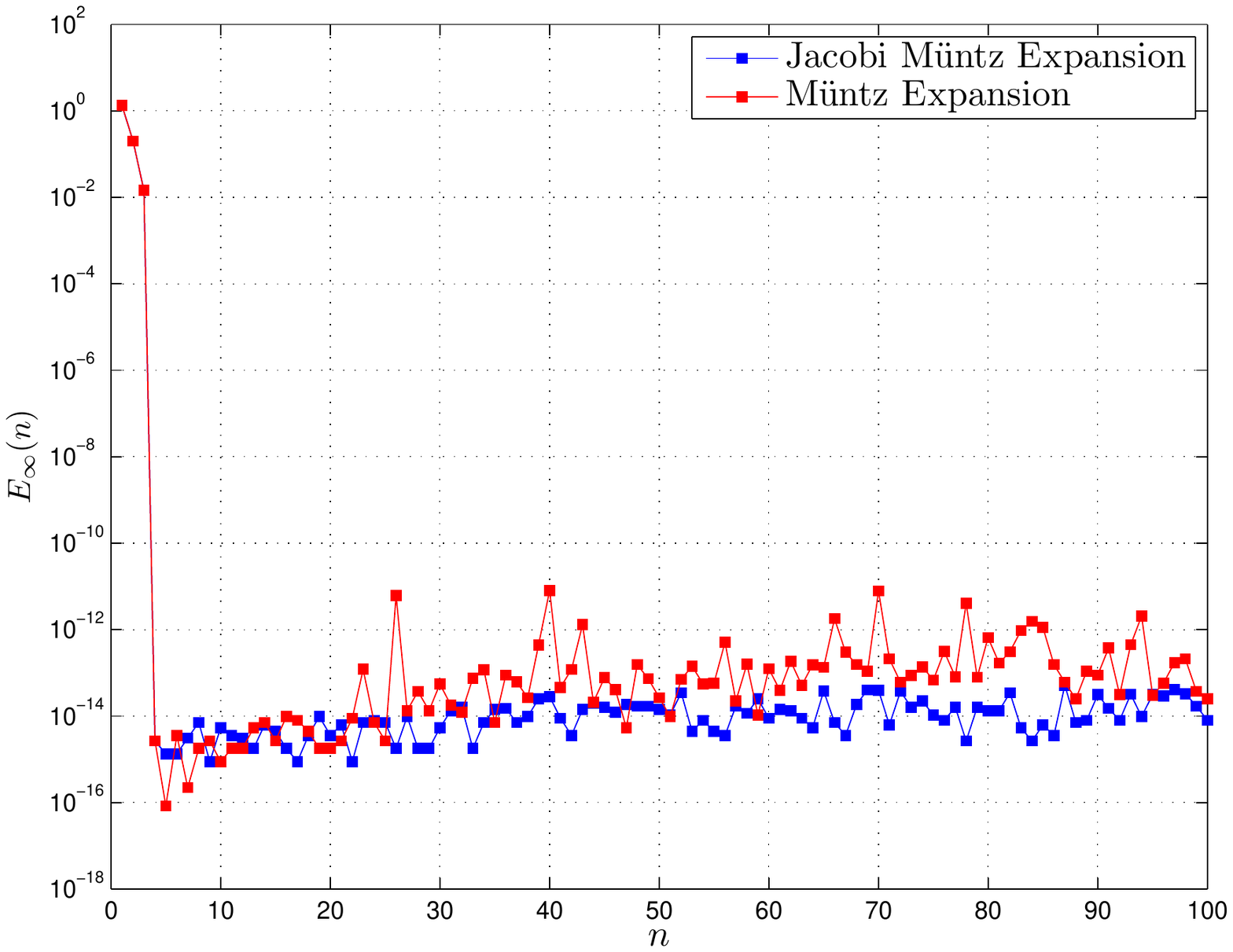}\includegraphics[width=6.5cm,height=8.5cm,keepaspectratio=true]{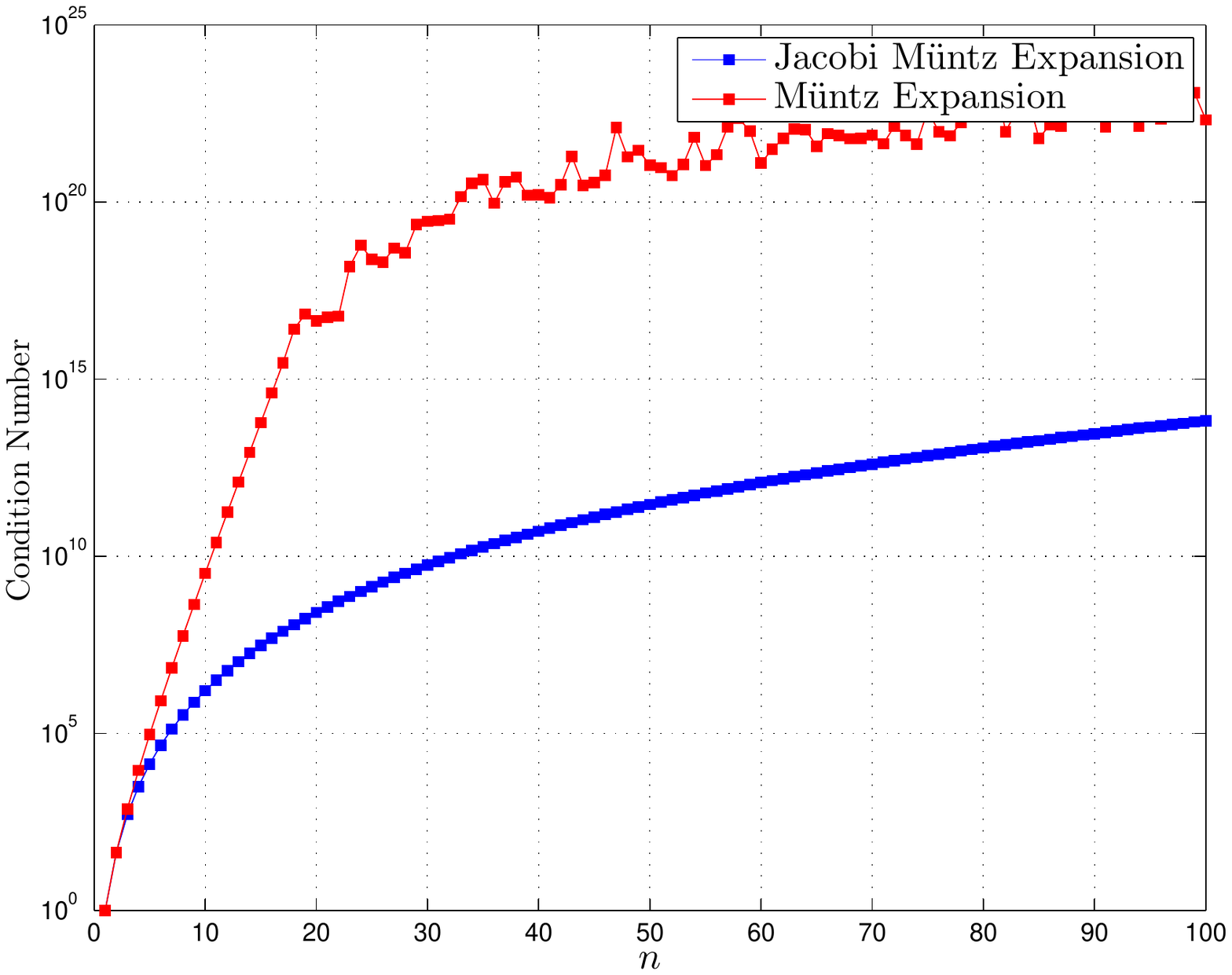}\\\vspace{-3.5cm}
		\centering
	\includegraphics[width=6.5cm,height=8.5cm,keepaspectratio=true]{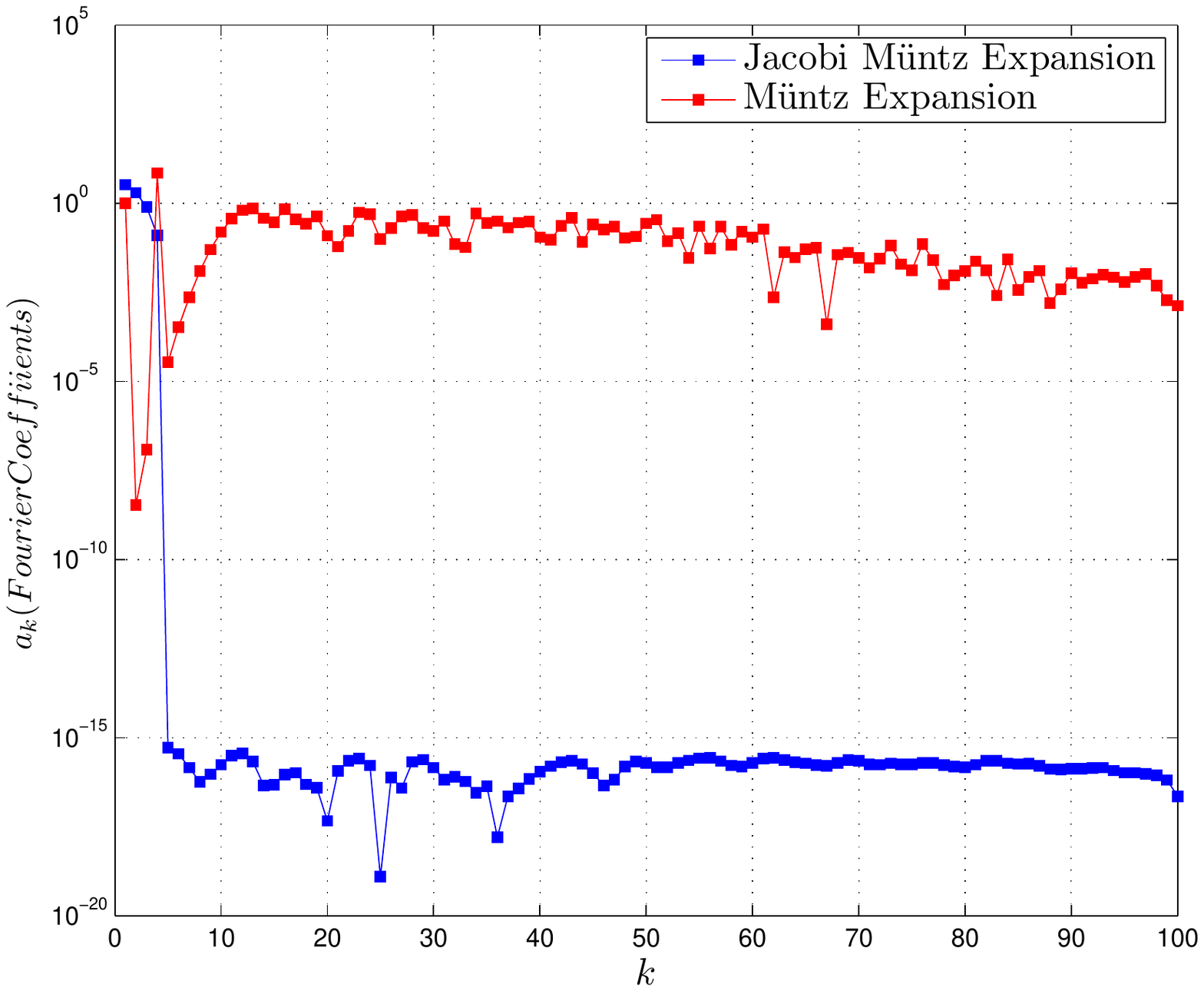}\vspace{-2.5cm}
	\caption{The first row: maximum errors together with the condition numbers (of the coefficient matrix) with $n=1:1:100,\ \alpha=0.5,\ \beta=1.5,\ \eta=-3,\ \sigma=0.5,\ \mu=1.5,\ b=1$ and $\nu=3$ in log-linear scale for \cref{Ex-1}. The second row: the Fourier coefficients $a_k,\  k=1:1:n$ with $n=100$ obtained by using the JMFs-1 and M\"untz basis functions for  $x\in[0,1]$.}
	\label{Fig-3}
\end{figure}

\end{example}

\subsection{Application to fractional partial differential equations}
\begin{example}\label{Ex-3}
	Consider the following fractional partial differential equation:
	\begin{eqnarray}\label{EqEx_3}
	\frac{\partial}{\partial t}u(x,t)&=& d(x,t)\  {}_{0}D_{x,\sigma,\eta}^{\mu}u(x,t)	+s(x,t),\ x\in[0,b],\ t\in[0,T], \label{EqEx_3-1} \\
	 u(0,t)&=& \frac{\partial}{\partial x}u(0,t)=0,\ u(0,x)=f(x),\   1<\mu<2, \label{EqEx_3-2}
	\end{eqnarray}
	where $u(x,t)$ is an unknown function and the functions $d(x,t)$ and $s(x,t)$ are arbitrary given functions.
	
	We start to approximate the unknown function $u(x,t)$ in problem \eqref{EqEx_3-1}-\eqref{EqEx_3-2} by $\tilde u_n(x,t)$ as follows:
	\begin{equation}
	u(x,t)\simeq \tilde u_n(x,t)=\sum_{k=0}^na_k(t)\ {}^1\mathcal{J}^{(\alpha,\beta,\mu,\sigma,\eta)}_k(x),
	\end{equation}  
	where the parameters $\beta,\ \mu,\ \eta$ are chosen such that $\tilde u_n(0,t)=\frac{\partial}{\partial x}\tilde u_n(0,t)=0$. Plugging $\tilde u_n(x,t)$ into \eqref{EqEx_3-1}-\eqref{EqEx_3-2} and collocating both sides at $\{x_j\}_{j=0}^n=\left\{x_j^{(\alpha,\beta,\sigma)}\right\}_{j=0}^n$ which is defined in \eqref{MuntsQuadNodWei} yield:
	\begin{subequations}\label{main-IVP-im}
		\begin{eqnarray}
		&&\mathbf M\,\dot{\mathbf{a}}(t)=\mathbf C(t)\  {\bf D^\mu} \ \mathbf {a}(t)+\mathbf{s}(t) ,\\
		&&\mathbf M\, \mathbf a(0)=\,\mathbf F,
		\end{eqnarray}
	\end{subequations}
		where ${\bf M}$ and ${\bf D^\mu}$ are defined in \eqref{MatM} and \eqref{MatDiffM}, respectively. Moreover, 
		\[
		\mathbf a(t)=\begin{bmatrix}
			a_0(t)\\
			a_1(t)\\
			\vdots\\
			a_n(t)
		\end{bmatrix},\ 
		\mathbf s(t)=\begin{bmatrix}
			s(x_0,t)\\
			s(x_1,t)\\
			\vdots\\
			s(x_n,t)
		\end{bmatrix},
		 \ \mathbf C(t)=\text{diag}\left(d(x_0,t),\dots,d(x_n,t)\right),\ \mathbf F=\begin{bmatrix}
		f(x_0)\\
		f(x_1)\\
		\vdots\\
		f(x_n)
		\end{bmatrix}.
		\]	
		Because of the fact that the matrix ${\bf M}$ is invertible, then we can rewrite \cref{main-IVP-im} as follows:
			\begin{subequations}\label{main-IVP-inverse}
				\begin{eqnarray}
				&&\,\dot{\mathbf{a}}(t)=\mathbf M^{-1}\left[\mathbf C(t)\  {\bf D^\mu} \ \mathbf {a}(t)+\mathbf{s}(t)\right] ,\\
				&&\mathbf a(0)=\,\mathbf M^{-1} \mathbf F,
				\end{eqnarray}
			\end{subequations}
			The obtained  system of ordinary differential equations is solved numerically by  \textsc{Matlab}'s  \textsc{ode45} routine with $\textsc{RelTol}=10^{-16},\ \textsc{AbsTol}=10^{-16}$. As a simple example, we take $u(x,t)=x^{\sigma\nu}\sin(t^2)$ and $\displaystyle d(x,t)=-\frac{1}{1+xt}$. In \cref{Fig-4}, approximate solution versus absolute error for $n=10,\ b=1,\ T=5,\ \alpha=0.5,\   \beta=3.5,\ \mu=1.5,\ \sigma=0.5,\ \eta=-1,\  \nu=7$.
	\begin{figure}[htbp]
		\vspace{-2.5cm}
		\centering
		\includegraphics[width=6.5cm,height=8.5cm,keepaspectratio=true]{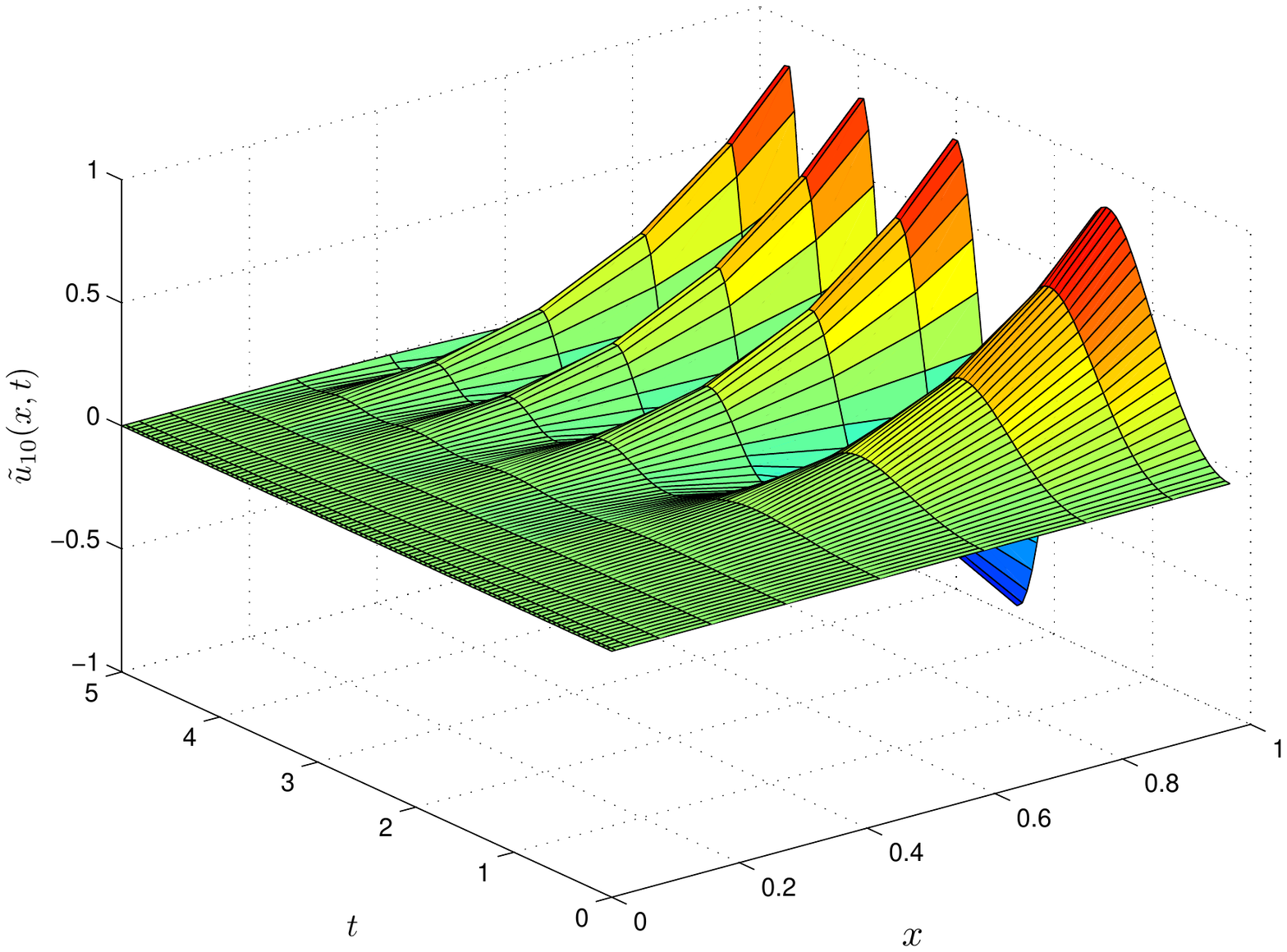}\includegraphics[width=6.5cm,height=8.5cm,keepaspectratio=true]{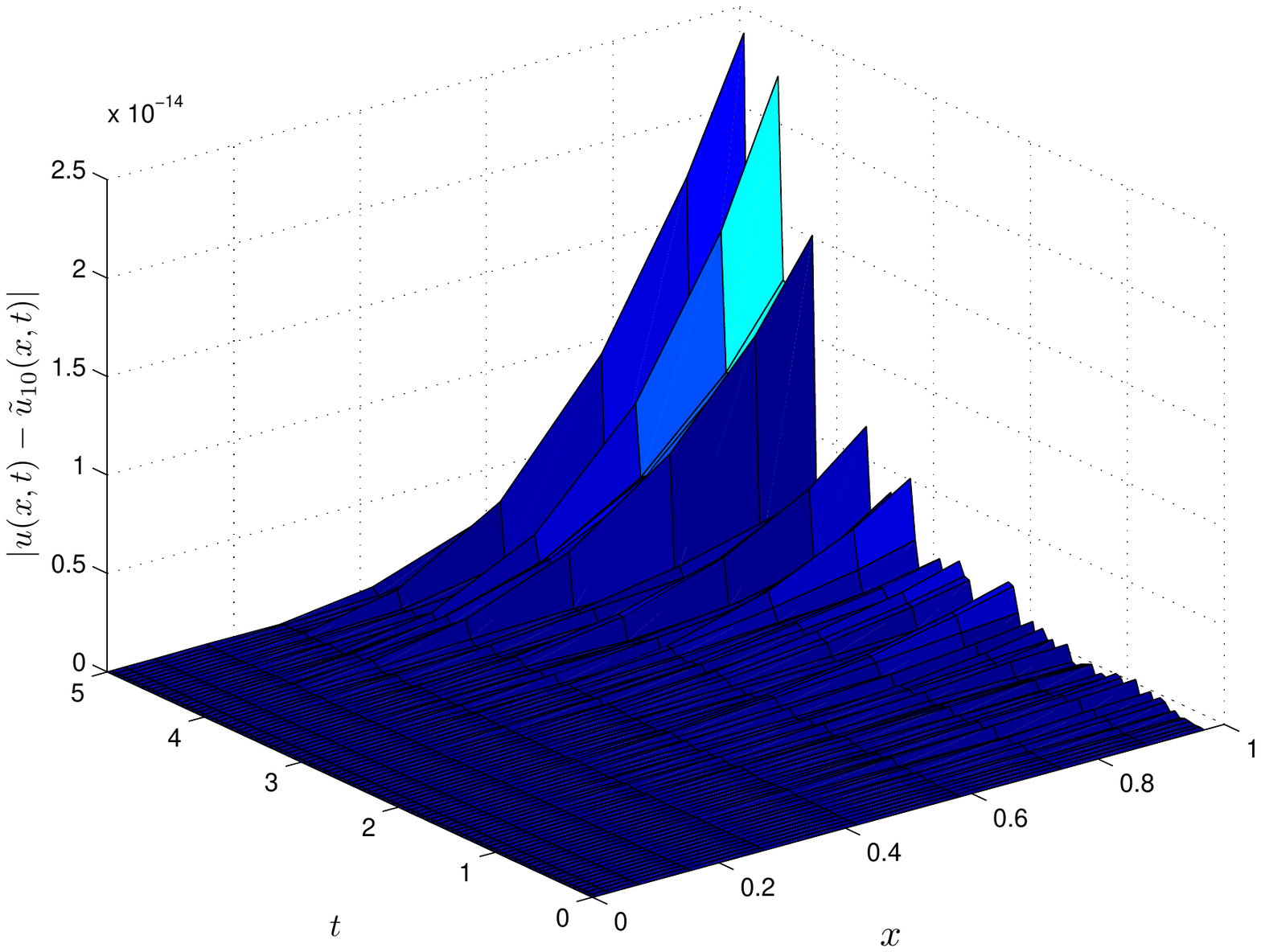}
		\vspace{-2.5cm}
		\caption{Behavior of the approximate solution (left) versus absolute error (right)  for $n=10,\ b=1,\ T=5,\ \alpha=0.5,\   \beta=3.5,\ \mu=1.5,\ \sigma=0.5,\ \eta=-1,\  \nu=7$.}
		\label{Fig-4}
	\end{figure}
	\end{example}
	\begin{example}
		The aim of this example is to show that the newly generated basis functions not only work for the nonlinear problem successfully but also can be used for the problem without fractional derivatives. For this purpose, we consider the well known Burgers' equation \cite{MR2867779}:
		\begin{eqnarray}\label{BURG}
		&&\frac{\partial u}{\partial t}=\epsilon \frac{\partial^2 u}{\partial x^2}-u\frac{\partial u}{\partial x}+s(x,t),\ \epsilon>0,\  x\in[0,b],\ t\in[0,T], \label{EqEx_4-1} \\
		&&u(0,t)= u(b,t)=0,\ u(0,x)=f(x). \label{EqEx_4-2}
		\end{eqnarray}
		Here, the following two different exact solutions are studied:
		\begin{eqnarray}
		u(x,t)&=&\left(\sqrt {1-\sqrt {x}}\right)\sqrt {x}\ \sin \left( \sqrt {x} \right) \cos \left( 
		{t}^{2} \right) 
		,\label{Burg1}\\
		u(x,t)&=&\left(\sqrt {1-\sqrt {x}}\right)\sqrt {x}\ \cos \left( \sqrt {x} \right) \cos \left( 
		{t}^{2} \right)
		,\label{Burg2}
		\end{eqnarray}
		over $(x,t)\in[0,1]\times[0,10]$. It is easy to verify that:
		\begin{itemize}
			\item Both exact solutions vanish at points $x=0$ and $x=1$.
			\item The first solution \eqref{Burg1} is non-smooth at $x=1$ and the second one \eqref{Burg2} is non-smooth at $x=0$ and $x=1$.
		\end{itemize}
		Due to the above facts, it is natural to use the JMFs-2 to approximate the exact solution over interval $[0,1]$ as follows:
			\begin{equation}
			u(x,t)\simeq \tilde u_n(x,t)=\sum_{k=0}^na_k(t)\ {}^2\mathcal{J}^{(\alpha,\beta,\sigma,\eta)}_k(x)
			=\sum_{k=0}^na_k(t)x^{\sigma\eta}\left(1-x^\sigma\right)^\alpha P_k^{(\alpha,\beta)}\left(2x^\sigma-1\right),
			\end{equation} 
		where $\sigma>0$ and the parameters $\eta,\ \alpha$ are chosen such that $\tilde u_n(0,t)=\tilde u_n(1,t)=0$.
		 Using the same fashion which addressed in previous example, we solve this problem numerically. Approximate solutions versus the error functions at $T=10$ for both cases \eqref{Burg1} and \eqref{Burg2} with some values of the parameters $n,\ b,\ T,\ \alpha,\   \beta,\ \eta,\ \sigma$ and $\epsilon$ are depicted in \cref{Fig-5} and \cref{Fig-6}, respectively. 
		 
		  It is observed from \cref{Fig-5} and \cref{Fig-6} that the approximate solutions based on the use of the JMFs-2 have good agreement with the exact ones. 
		  
		  An important point which must be emphasized here is that although the exact solutions \eqref{Burg1} and \eqref{Burg2} are  non-smooth on $[0,1]$ but as it can be clearly seen in \cref{Fig-7} the error functions $E_2(n)$ ($L^2$-norm)  and $E_\infty(n)$ ($L^{\infty}$ norm) decay exponentially for large values of  $n$.
			\begin{figure}[htbp]
				\vspace{-2.5cm}
				\centering
				\includegraphics[width=6.5cm,height=8.5cm,keepaspectratio=true]{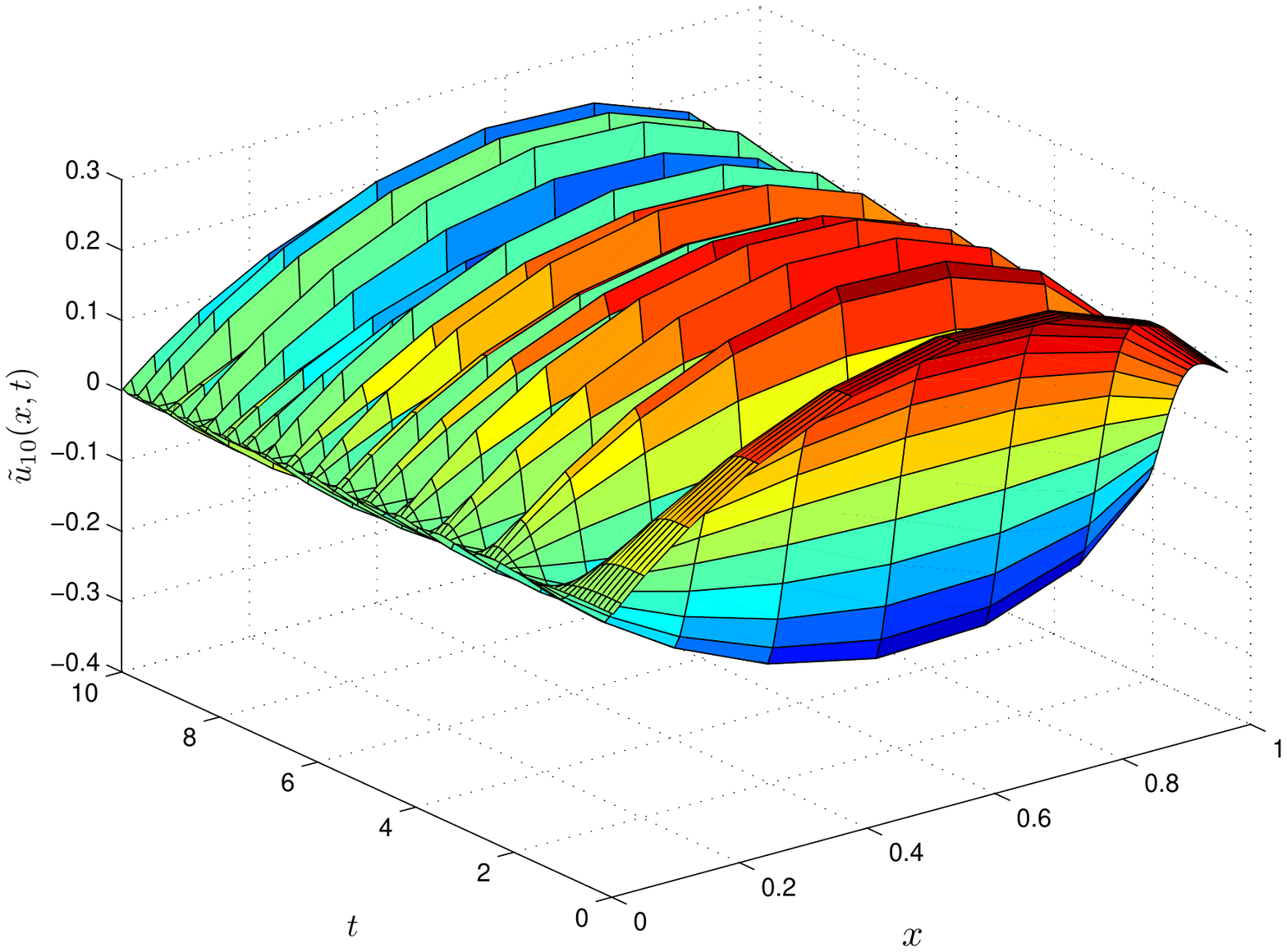}\includegraphics[width=6.5cm,height=8.5cm,keepaspectratio=true]{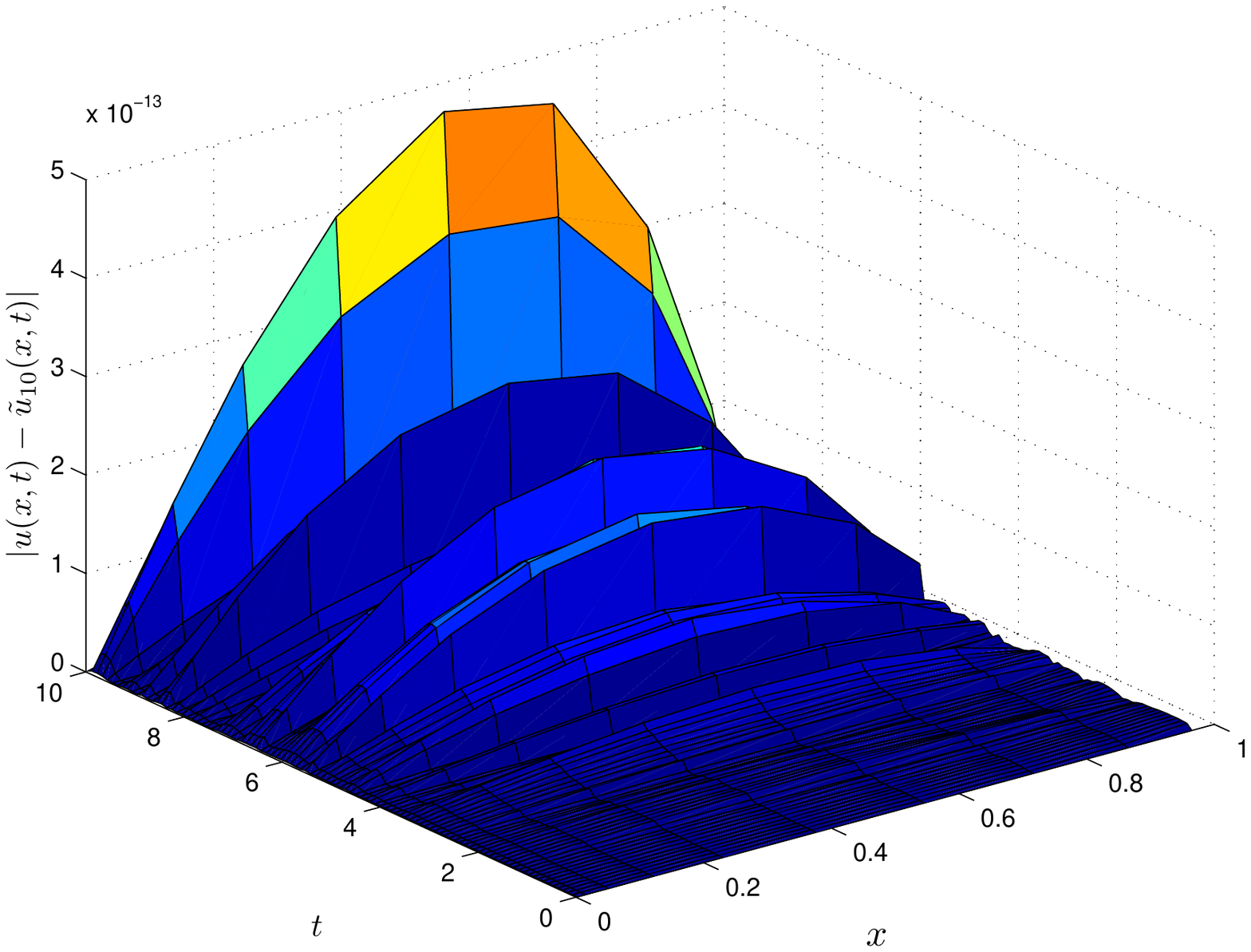}
				\vspace{-2.5cm}
				\caption{Behavior of the approximate solution (left)  versus the absolute error (right)  for $n=10,\ b=1,\ T=10,\ \alpha=0.5,\   \beta=2,\ \eta=2,\ \sigma=0.5,\ \epsilon=0.1$ for the exact solution \eqref{Burg1}.}
				\label{Fig-5}
			\end{figure}
				\begin{figure}[htbp]
					\vspace{-2.5cm}
					\centering
					\includegraphics[width=6.5cm,height=8.5cm,keepaspectratio=true]{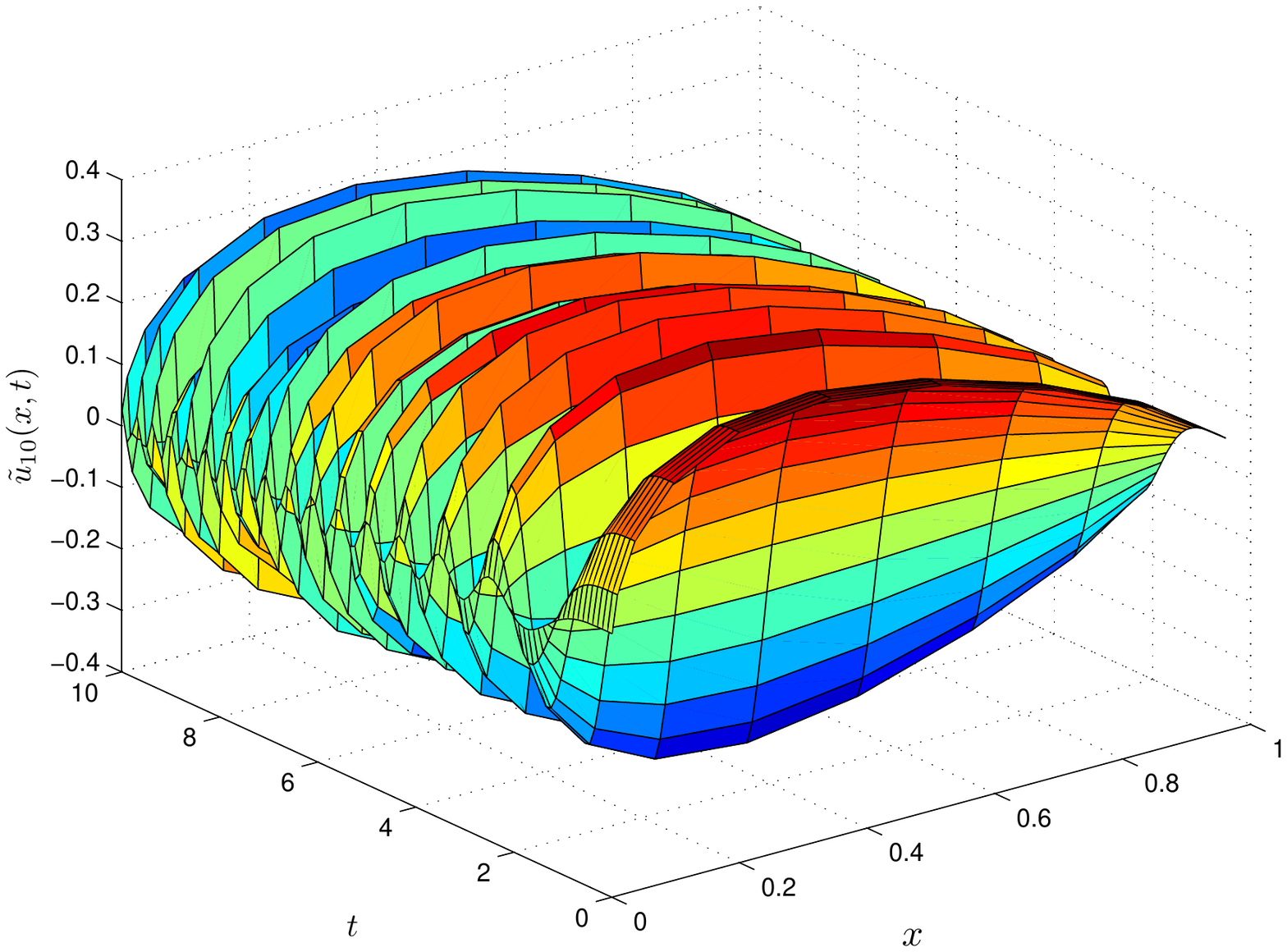}\includegraphics[width=6.5cm,height=8.5cm,keepaspectratio=true]{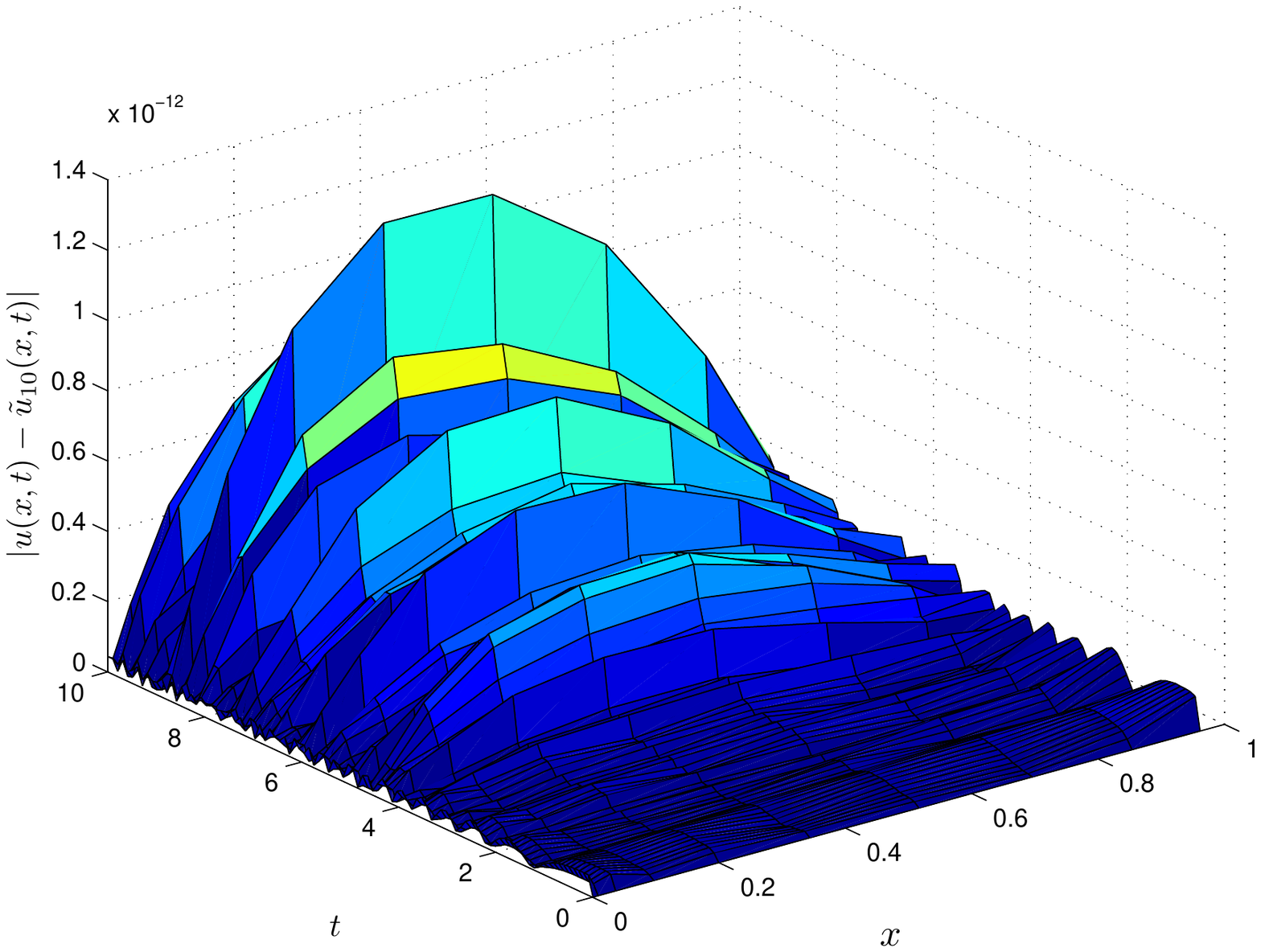}
					\vspace{-2.5cm}
					\caption{Behavior of the approximate solution (left) versus the absolute error (right) for $n=10,\ b=1,\ T=10,\ \alpha=0.5,\   \beta=1,\ \eta=1,\ \sigma=0.5,\ \epsilon=0.1$ for the exact solution \eqref{Burg2}.}
					\label{Fig-6}
				\end{figure}
					\begin{figure}[htbp]
						\vspace{-2.5cm}
						\centering
						\includegraphics[width=6.5cm,height=8.5cm,keepaspectratio=true]{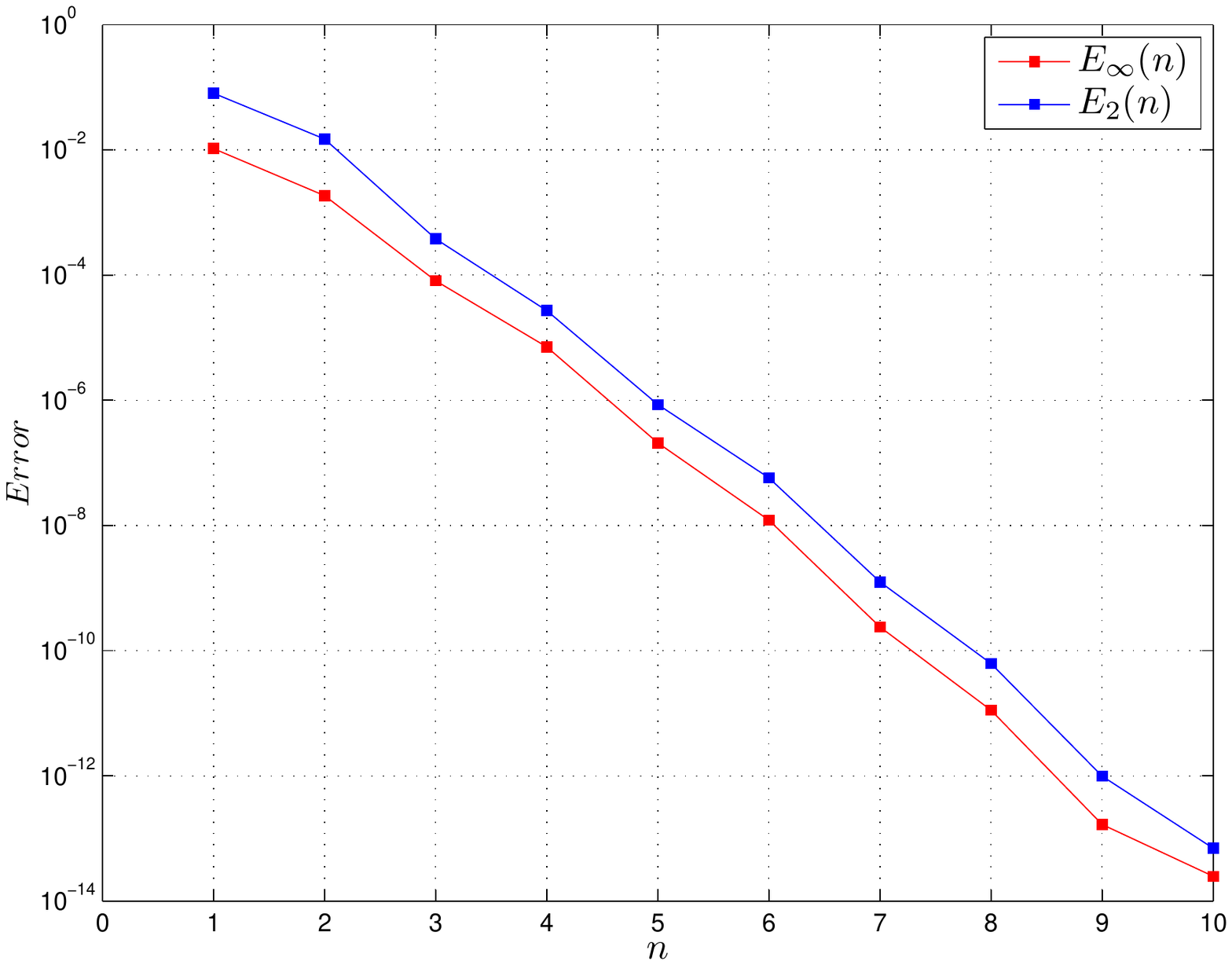}\includegraphics[width=6.5cm,height=8.5cm,keepaspectratio=true]{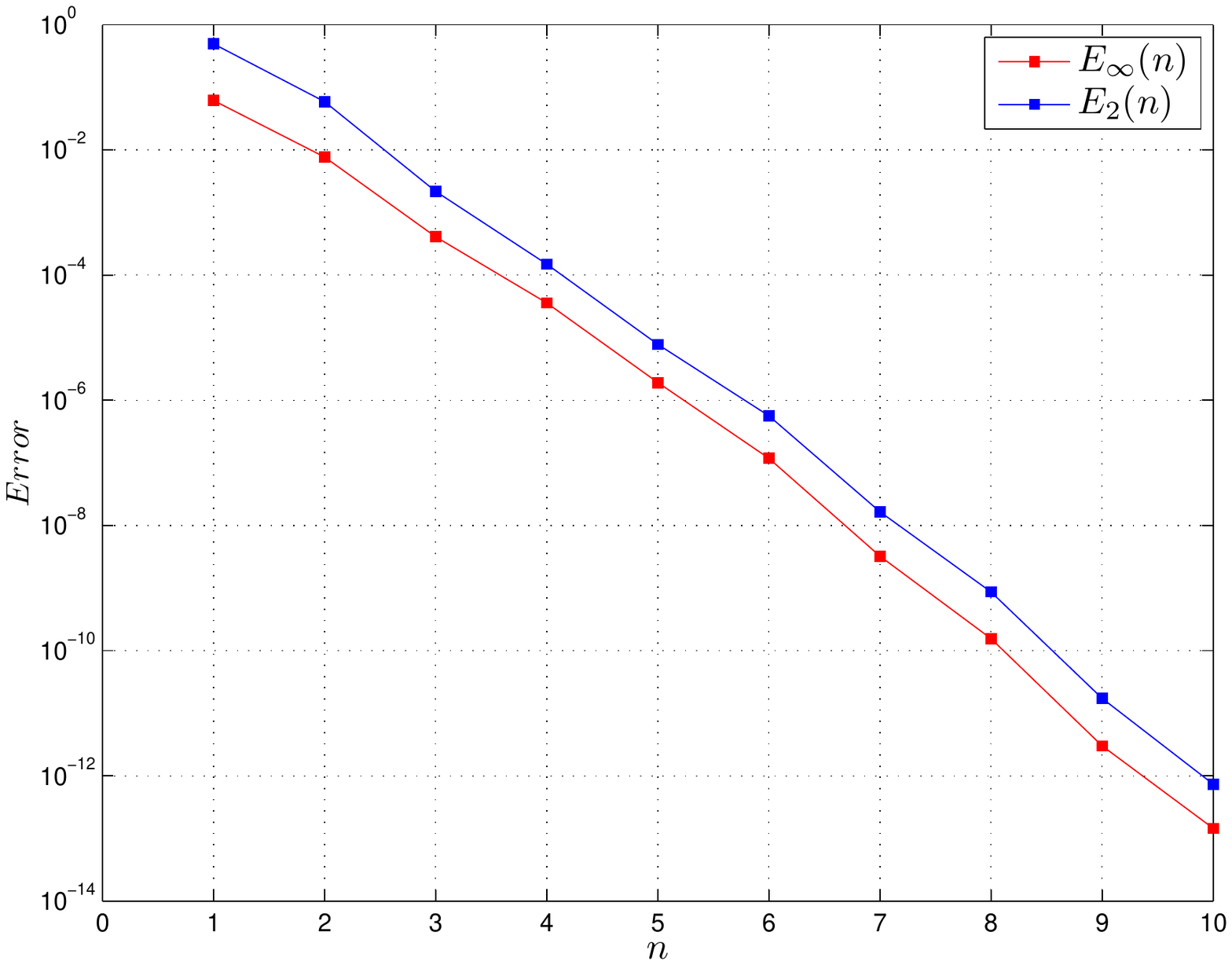}
						\vspace{-2.5cm}
						\caption{The error functions $E_2(n)$  and $E_\infty(n)$ for the exact solutions \eqref{Burg1} (left) and \eqref{Burg2} (right), for various values of $n$   with the values $ b=1,\ T=10,\ \alpha=0.5,\   \beta=2,\ \eta=2,\ \sigma=0.5,\ \epsilon=0.1$.}
						\label{Fig-7}
					\end{figure}
		\end{example}
\section{Conclusions}
\label{sec:conclusions}
Two new classes of M\"untz functions  which we called Jacobi-M\"untz functions are introduced and their useful properties are provided in detail. With some numerical examples, the efficiency and accuracy of these basis functions are verified. It  can be easily observed from the numerical results that the newly generated basis functions are used to establish new spectral collocation methods with exponential accuracy for the problems with non-smooth solutions.  

The authors believe that this article opens a new window for future researches in the filed of numerical analysis. 

	At the end, we address some future works which are in the continuation of this research:
	\begin{itemize}
		\item  With respect to the new (modal) basis functions JMFs-1 and JMFs-2, two new nodal basis functions (with cardinality property) can be defined which are, in fact, the generalized forms of the classical Lagrange basis polynomials.  The pseudo spectral, discontinuous Galerkin and finite element methods and generally other nodal based methods can be developed similarly. 
		\item These basis functions can be used in other modal based (or projection) methods such as: Galerkin, tau and Petrov-Galerkin methods and etc.
		\item Other classes of the  M\"untz Sturm-Liouville problems can be introduced and their properties can be studied in detail.
	\end{itemize}


\end{document}